\documentclass[11pt,final]{iopart}
\usepackage{amssymb}
\usepackage{amsfonts}
\usepackage{amsthm}
\usepackage{iopams}
\usepackage{graphicx,subfigure}
\usepackage[dvips]{epsfig}
\usepackage{float}
\usepackage[bottom]{footmisc}

\usepackage{perpage}
\MakePerPage{footnote}
\usepackage[title,titletoc,toc]{appendix}
\usepackage{prettyref}
\usepackage{appendix}
\usepackage{wrapfig}
\usepackage{showkeys}
\usepackage{todo}
\usepackage{prettyref}
\usepackage{verbatim}
\usepackage[latin1]{inputenc}
\usepackage{subfigure}
\usepackage[dvips]{epsfig}
\usepackage[bottom]{footmisc}
\usepackage{algorithm}
\usepackage{algpseudocode}
\usepackage{verbatim}
\newrefformat{eq}{(\ref{#1})}
\newrefformat{fig}{Figure~\ref{#1}}
\usepackage[pdfpagelabels]{hyperref}
\hypersetup{colorlinks=true,
linkcolor=blue,
citecolor=blue,
plainpages=false}
\newtheorem{theorem}{\sc Theorem}[section]

\newtheorem{ppty}[theorem]{\sc Property}
\newtheorem{lemma}[theorem]{\sc Lemma}
\newtheorem{assump}[theorem]{\bf Assumption}

\newtheorem{definition}[theorem]{\sc Definition}
\newtheorem{remark}[theorem]{\sc Remark}

\newtheorem{cor}[theorem]{\sc Corollary}

\newcommand{\argmin}{{\rm arg\,min}}

\eqnobysec

\begin{document}

%
\title[The Parameters in a Nested Primal-Dual Algorithm]
{Choice of the Parameters in A Primal-Dual Algorithm 
for Bregman Iterated Variational Regularization}

%
\author{Erdem Altuntac$^{1}$
\footnote{ Major part of this work has been done 
within the framework of ARC grant at Universit\'e Libre de 
Bruxelles during author`s PostDoc research period 2017 - 2019.}}

\address{$^{1}$Fraunhofer Institute for High Frequency 
Physics and Radar Techniques (FHR), Wachtberg, Germany }


\ead{\mailto{erdem.altuntac@fhr.fraunhofer.de}}

\begin{abstract}

Focus of this work is solving a 
non-smooth constraint minimization problem by a 
primal-dual splitting algorithm involving proximity operators.
The problem is penalized by the
Bregman divergence associated with the 
non-smooth total variation (TV) functional.

We analyse two aspects: Firstly, the convergence of the 
regularized solution of the minimization 
problem to the minimum norm solution. Second, the 
convergence of the iteratively regularized minimizer  
to the minimum norm solution by a primal-dual algorithm.
For both aspects, we use the assumption of a variational 
source condition (VSC).
This work emphasizes the impact of the
choice of the parameters in stabilization of a
primal-dual algorithm.
Rates of convergence are obtained in terms of some concave, 
positive definite index function. 

The algorithm is applied to a simple two dimensional 
image processing problem. Sufficient error analysis profiles
are provided based on the size of the forward operator 
and the noise level in the measurement.

\end{abstract}

\bigskip


\section{Introduction}

One can only stabilize an algorithm for correctly defined
parameters. Otherwise, no regularization technique 
provides approximate solution for an inverse ill-posed problem.
This work does not only focus on stability analysis of a convex optimization
algorithm, but it also introduces the algorithm as an iterative regularization
procedure. The algorithm we take up is in the generalization
of gradient descent by means of proximal mapping. 
In the study of convergence of a gradient descent algorithm,
it is known that the stability of the algorithm is in fact based on how
the step-length is defined, \textit{e.g.} \textbf{\cite{BeckTeboulle09}}.
By stability, we mean convergence of the iterative solution to the
minimizer of some appropriately defined objective functional.
However, from regularization point of view, this is an insufficient
stability analysis for the study of inverse ill-posed problems 
aims to provide answer for real life application problems. 
Therefore, in order to fill this gap in the context of 
variational regularization,
this work aims to go further than just defining 
a step-length for the stability of a convex optimization 
algorithm.

In general terms, regularization theory deals with the approximation
of some ill-posed inverse problems by a family of parametrized 
well-posed problems. Traditional quadratic-Tikhonov regularization
\textbf{\cite{Tikhonov63, TikhonovArsenin77}} has been
well established and analyzed \textbf{\cite{Engl96}}.
As an alternative to classical Tikhonov regularization,
convex variational regularization with some general
penalty functional has gained interest over the last decade. 
Analysis of convex variational regularization has been motivated
by a new image denoising method called \textit{``total variation``} 
\textbf{\cite{RudenOsherFatemi92}}.
Further studies of the method have been widely carried
out in the communities of inverse problems and optimization 
\textbf{\cite{AcarVogel94, BachmayrBurger09, BardsleyLuttman09, 
ChambolleLions97,
ChanChen06, ChanGolubMulet99, DobsonScherzer96, 
DobsonVogel97, VogelOman96}}. 

Recently proposed non-smooth penalty terms have risen
the interest in the study of variational regularization. 
Convexity of the objective functional is the most essential property
to analyze the stability of the regularized
solution of the inverse problem, or equivalently the minimization problem. 
Formulating the minimization problem 
as variational problem and estimating convergence rates under 
the assumption that the minimum norm solution
satisfies a type of \textit{variational source condition} (VSC)
has been well established, \textbf{\cite{BurgerOsher04, Grasmair10, 
Grasmair13, GrasmairHaltmeierScherzer11, HofmannScherzer07, Lorenz08}}.
In addition to the estimation of the convergence rates, 
verification of VSC has also become popular, 
see \textbf{\cite{HohageWeidling15, HohageWeidling16, SprungHohage17}}.
A recent study on the existence of the VSCs for linear or non-linear problems
can be found in \textbf{\cite{Flemming18}}.

When obtaining the stable regularized solution, it is also
important that this solution meets the constraints
of the inverse problems. These contraints are defined mostly
due to the physical facts of the solution. Therefore, the necessity 
of solving a constrained minimization problem is well understood.
We are tasked with approximating the solution of some constrained
minimization problem by an efficient proximal gradient algorithm
as an iterative regularization method. 

Inverse problems arise in many scientific fields: X-ray computed 
tomography, image processing, signal processing, wave scattering, 
shape reconstruction, etc. We must emphasize that although
many problems lie in the same family of inverse problems, 
each problem may differ from each other depending on their 
physical or engineering facts.
The properties of the targeted data and the forward operator
that we define in the following section are 
suitable for tomographic reconstruction problems.


\section{Notations and Mathematical Setting}
\label{notations}

Over the finite dimensional Hilbert spaces $\mathcal{X} = \mathbb{R}^{N}$
and $\mathcal{Y} = \mathbb{R}^{M}$, let us assume some linear, injective,
forward operator $T : \mathcal{X} \rightarrow \mathcal{Y}.$ We 
consider the following linear ill-posed problem
\begin{eqnarray}
\label{inverse_problem}
\delta\xi + T u = v^{\delta},
\end{eqnarray}
where the given noisy data $v^{\delta} \in \mathcal{Y}$ and 
the noise model $\xi$ with the noise
magnitude $\delta.$
Furthermore, we also impose non-negativity constraint 
on our targeted data $u.$
Then, the constraint domain $\Omega : = 
\{u \in \mathcal{X} : u_i \geq 0 
\mbox{ for each } i = 1, \cdots, N \},$
is treated as the indicator function, 
\textbf{\cite{BorweinLuke11}}, that is defined by
\begin{eqnarray}
\label{indicator}
h(u) = 1_{\Omega}(u) := \left\{ \begin{array}{rcl}
0 & \mbox{, for} & u \in \Omega \subset \mathcal{X} \\ 
\infty  & \mbox{, for} & u \notin \Omega \subset \mathcal{X} .
\end{array}\right.
\end{eqnarray}

Throughout the work, unless otherwise stated,
the notation $\Vert \cdot \Vert$ without any subscript
will be used to denote the usual Euclidean norm.
Let $\sigma(T^{\ast}T):= \{\sigma_1, \sigma_2, \cdots \sigma_M \}$  
be the spectrum of $T^{\ast}T.$
Then, for the finite dimensional forward operator 
$T : \mathbb{R}^{N} \rightarrow \mathbb{R}^{M}$ 
where $M < N,$ we define
\begin{displaymath}
\Vert T\Vert := \max_{1\leq k \leq M} 
\left\{ \sqrt{\sigma_{k}} \right\}.
\end{displaymath}

Regarding the stability analysis for the algorithm,
we will refer to one fundamental equality that
has been also given in \textbf{\cite[Eq. (2.1)]{Takahashi13}}.
For some $u_1,u_2 \in \mathcal{X}$ and $\rho \in \mathbb{R},$ 
the following equality holds, 
\begin{eqnarray}
\label{strct_cvx_eq}
\Vert\rho u_1 + (1-\rho)u_2\Vert^2 = \rho\Vert u_1 \Vert^2 + 
(1-\rho)\Vert u_2\Vert^2 -\rho(1-\rho)\Vert u_1 - u_2\Vert^2.
\end{eqnarray}

For some real valued and convex function 
$f : \mathcal{X} \rightarrow \mathcal{X}$ and some point 
$x$ in the domain of $f,$ the {\em subdifferential} of $f$ at 
$x^{\prime},$ denotes $\partial f(x^{\prime})$ is defined by
\begin{eqnarray}
\label{def_subdiff}
\partial f(x^{\prime}) := \left\{ \eta \in \mathcal{X}^{\ast} :
f(x) - f(x^{\prime}) \geq \langle \eta , x - x^{\prime} \rangle
\mbox{ for all } x \in \mathcal{X} \right\}.
\end{eqnarray}

\begin{definition}\textbf{[Generalized Bregman Distance]}
Let $\mathcal{J} : \mathcal{X} \rightarrow \mathbb{R}_{+} \cup \{ \infty \}$ 
be a convex functional 
with the subgradient $q^{\ast} \in \partial \mathcal{J}(u^{\ast}).$ 
Then, for $u, u^{\ast} \in \mathcal{X},$ the Bregman distance 
associated with the functional $\mathcal{J}$ is defined by
\begin{eqnarray}
\label{bregman_divergence_intro}
D_{\mathcal{J}} : & \mathcal{X} \times \mathcal{X} & 
\longrightarrow \mathbb{R}_{+}
\nonumber\\
& (u , u^{\ast}) & \longmapsto D_{\mathcal{J}}(u , u^{\ast}) 
:= \mathcal{J}(u) - \mathcal{J}(u^{\ast}) - 
\langle q^{\ast} , u - u^{\ast} \rangle .
\end{eqnarray}
It is well known that the Bregman distance does not satisfy
symmetry,
\begin{displaymath}
D_{\mathcal{J}}(u , u^{\ast}) \neq D_{\mathcal{J}}(u^{\ast} , u),
\end{displaymath}
and for the defined convex functional $\mathcal{J},$
\begin{displaymath}
D_{\mathcal{J}}(u , u^{\ast}) \geq 0.
\end{displaymath}
\end{definition}

Over the past decades, variational and traditional
regularization strategies have been dedicated to find
the stable minimum of the generalized form of the Tikhonov 
functional
\begin{eqnarray}
\label{obj_functional0}
H_{\alpha}: & \mathcal{X} & \times \mathcal{Y} \longrightarrow 
\mathbb{R}_{+}
\nonumber\\
& (u &, v^{\delta}) \longmapsto H_{\alpha}(u,v^{\delta}) := 
\frac{1}{2\alpha} \Vert Tu - v^{\delta}\Vert^2 + \mathcal{J}(u),
\end{eqnarray}
with a convex, lower-semicontinuous, not necessarily smooth
penalty functional
$\mathcal{J} : \mathcal{X} \longrightarrow 
\mathbb{R}_{+} \cup \{ \infty \}$ and
a regularization parameter $\alpha > 0.$ On the other 
hand, we consider the following objective functional,
\begin{eqnarray}
\label{obj_functional2}
F_{\alpha}: & \mathcal{X} & \times \mathcal{Y} 
\longrightarrow \mathbb{R}_{+}
\nonumber\\
& (u &, v^{\delta}) \longmapsto F_{\alpha}(u,v^{\delta}) := 
\frac{1}{2} \Vert Tu - v^{\delta}\Vert^2 +  
\alpha D_{\mathcal{J}}(u , u^{0}) + h(u),
\end{eqnarray}
with some initial estimation $u^{0} \in \mathcal{X},$ 
where we have included the indicator of the positive 
orthant. In particular, we
associate the Bregman distance penalty term with
the total variation (TV) functional defined by
\begin{eqnarray}
\label{TV_penalty}
TV(u, \Omega) = \mathcal{J}(u) := \int_{\Omega} 
\vert \nabla u(x) \vert dx 
\approxeq \sum_{i} \vert \nabla_{i} u \vert ,
\end{eqnarray}
which is, in the 3D case, $i = (i_x, i_y, i_z).$ 
For some $u_1, u_2$ in the domain of $T,$ 
from Lipschitz continuity of the misfit term 
$\frac{1}{2} \Vert Tu - v^{\delta}\Vert^2$
that is, 
\begin{eqnarray}
\Vert T^{T}T(u_{1} - u_{2})\Vert^2 & \leq & 
\Vert T\Vert^2\Vert T(u_{1} - u_{2})\Vert^2
\nonumber\\
& = &\Vert T \Vert^2\langle T(u_{1} - u_{2}) , 
T(u_{1} - u_{2}) \rangle
\nonumber\\
& = & \Vert T \Vert^2\langle T^{T}T(u_{1} - u_{2}) , 
u_{1} - u_{2} \rangle , 
\nonumber
\end{eqnarray}
one can easily observe the following,
\begin{eqnarray}
\label{LipschitzEstMisfit}
-\langle u_{1} - u_{2} , 
T^{T}T(u_{1} - u_{2}) \rangle \leq 
-\frac{1}{\Vert T\Vert^2}\Vert T^{T}T(u_{1} - u_{2})\Vert^2.
\end{eqnarray}

For the sake of following further calculations 
easily in future developments of this work, 
we introduce TV in the composite form
\begin{eqnarray}
\label{composite_TV}
J(u) = g(D u) \mbox{ where, }
g(\cdot) = \Vert \cdot\Vert_{1} \mbox{ with }
D(\cdot) = \nabla (\cdot).
\end{eqnarray}
Thus,
\begin{eqnarray}
\label{subdiff_TV}
\partial J(u) = D^{\ast} \partial g(Du).
\end{eqnarray}
Total variation type regularization targets 
the reconstruction of bounded variation (BV) 
class of vectors, which are functions 
in the infinite dimensional mathematical setting, 
that are defined by
\begin{eqnarray}
\label{bv_def}
BV(\Omega) := \{ u \in \ell^{1}(\Omega) : 
TV(u , \Omega) < \infty \}
\end{eqnarray}
with the norm
\begin{eqnarray}
\label{BV_norm_def}
\Vert\varphi\Vert_{BV} := \Vert u \Vert_{1} + TV(u , \Omega).
\end{eqnarray}
Let the mean value 
$MV : \ell^{1}(\Omega) \rightarrow \mathbb{R}$ 
be defined by
\begin{eqnarray}
\label{mean_value_funct}
MV[u] := \frac{1}{\vert \Omega \vert} \int_{\Omega} u(x) dx .
\end{eqnarray}
It has been stated in \textbf{\cite[Eq. (4.3)]{AcarVogel94}} that, 
over the bounded domain $\Omega,$ 
for any $u \in BV(\Omega)$ has the following decomposition
\begin{eqnarray}
\label{BV_decomp1}
u = \tilde{u} + MV[u] \vec{1},
\end{eqnarray}
where 
\begin{eqnarray}
\label{BV_decomp2}
MV[\tilde{u}] = 0.
\end{eqnarray}
\section{Discussion of Previous Works and Contribution}

In an early study \textbf{\cite{OsherBurger05}}, 
Bregman iteration has been
proposed to solve the basis pursuit problem. 
Therein, it has been numerically demonstrated that the efficiency 
of the algorithms gets improved with the inclusion of the Bregman 
distance playing the role of
penalty term in the objective functional $F_{\alpha}$ in 
(\ref{obj_functional2}).
In a recent study by Sprung and Hohage \textit{et al.}, 2017,
\textbf{\cite{SprungHohage17}}, it has been stated that minimizing the
objective functionals in the form of (\ref{obj_functional0})
can be interpreted as proximal point methods 
when one considers the penalty term as
a quadratic function. 
Applying optimization algorithms in the field of inverse problems
as iterative regularization method has become popular. Authors in 
\textbf{\cite{GarrigosVilla18}} have proposed some primal-dual
algorithm, wherein the convergence has been studied for the given
noiseless measurement data. Different forms of nested primal-dual algorithms 
for solving proximal mappings have been introduced in 
\textbf{\cite{ChenLoris18}}.

We consider linear, inverse 
ill-posed problems in the general form (\ref{inverse_problem}) 
and study the stability
of both iterative and non-iteartive regularized solutions in the context
of convex variational regularization. The main results of our work
are derived in the presence of noisy measurements and specifically address
 convex variational regularization. From the subdifferential
characterization of the regularized minimizer for the functional
(\ref{obj_functional2}), a new iterative
regularization algorithm shall be developed. In Section \ref{primal-dual},
stability analysis of the iteratively regularized solution
is analyzed in the Hadamard sense.

\section{Overview of the Fundamentals of the Regularization Theory}
\label{regularization_overview}

We split this section into three subsections. Section 
\ref{cont_reg} reviews the regularization theory 
in the continuous sense. This part will be used to analyse 
the convergence of the regularized
minimum $u_{\alpha}^{\delta}$ towards the minimum norm solution
$u^{\dagger}.$ A choice of the regularization parameter which is
\textit{a-posteriori} obeying Morozov`s dicrepancy principle 
is introduced in the following up sections.
Section \ref{subsection_iter_reg}, on the other hand, reviews 
the iterative regularization theory which will be used for 
showing the convergence of the iteratively regularized minimum 
$u_{i+1}$ still towards the minimum norm solution $u^{\dagger}.$

\subsection{Continuous regularization}
\label{cont_reg}
In the Banach space setting, the concept of 
\textbf{$\mathcal{J}$-minimizing solution} is well known for the
functionals in the form of (\ref{obj_functional0}), 
see {\em e.g.,}\textbf{\cite[Lemma 3.3]{Schuster12}, 
\cite{Flemming18}}. 
In our case, since Bregman distance is associated with the TV
functional, we have
\begin{eqnarray}
\label{J-min_est}
T u^{\dagger} = v^{\dagger} \mbox{ and }
\Vert u^{\dagger} \Vert = 
\min \{\Vert u \Vert : u\in \mathrm{BV}(\Omega) , 
T u = v^{\dagger} \}.
\end{eqnarray}
According to our mathematical setting, the functional
$\mathcal{J} : \Omega \subset \mathcal{X} \rightarrow \mathbb{R}_{+}$ 
attains some finite value
only at some finite point $u \in \mathrm{BV}(\Omega),$ 
\textbf{\cite[Assumption 1.1 (i)]{Flemming18}}. Moreover, 
our linear forward operator is defined on a 
uniformly convex Banach space
\textbf{\cite[Theorem 2.53(k) \& Lemma 3.3]{Schuster12}}.
In case of $u^{0}$ to be constant, then from our setting
above (\ref{composite_TV}) and (\ref{subdiff_TV}), our notation in
(\ref{J-min_est}) boils down to its conventional form 
named $\mathcal{J}$-minimizing solution. In what follows, 
the minimum norm solution that has just been introduced by 
(\ref{J-min_est}) will be denoted by $u^{\dagger}.$

Briefly speaking, establishing convergence result for
some regularization method in the Hadamard sense begins with 
seeking to approximate the true solution by a family
of regularized solutions of the problem
\begin{eqnarray}
\label{problem0}
u_{\alpha}^{\delta} \in \argmin_{u \in \mathcal{X}} F_{\alpha} ,
\end{eqnarray}
satisfying the following properties:
\begin{enumerate}
\item For any $v^{\delta} \in \mathcal{Y}$ there exists a solution 
$u_{\alpha}^{\delta}\in \mathcal{X}$ to the problem (\ref{problem0});
\item For any $v^{\delta} \in \mathcal{Y}$ there is no more than one 
$u_{\alpha}^{\delta}\in \mathcal{X};$
\item Convergence of the regularized solution 
$u_{\alpha}^{\delta}$ 
to the minimum norm solution $u^{\dagger}$ must continuously 
depend on the given data $v^{\delta}$, \textit{i.e.}
\begin{displaymath}
\Vert u_{\alpha}^{\delta} - u^{\dagger}\Vert 
\rightarrow 0 \mbox{ as } \alpha = \alpha(\delta , v^{\delta}) 
\rightarrow 0 \mbox{ for } \delta \rightarrow 0
\end{displaymath}
whilst
\begin{displaymath}
\Vert v^{\dagger} - v^{\delta}\Vert \leq \delta
\end{displaymath}
where $v^{\dagger} \in \mathcal{R}(T) \subset \mathcal{Y}$ 
is the noiseless measurement
and $\delta$ is the noise level.
\end{enumerate}
Existence and uniqueness of the regularized minimizer for the functional
(\ref{obj_functional2}) is ensured by the following facts: 
Bregman distance is convex and lower semi-continuous in its first
term, and so is the indicator function.
In (iii) it is stated that when the given measurement $v^{\delta}$
lies in some $\delta-$ball centered at the noisless measurement 
$v^{\dagger}$,
\textit{i.e.} $v^{\delta}\in\mathcal{B}_{\delta}(v^{\dagger}),$ 
then the regularized solution $u_{\alpha}^{\delta}$ must 
converge to the minimum norm solution of the inverse problem 
$u^{\dagger}$ as the regularization parameter 
$\alpha(\delta , v^{\delta})$ 
decays sufficiently, see 
\textbf{\cite[Section 2, Properties (2.1) - (2.3)]{Engl96}}.
The main results of this work are dedicated to justify the condition
`(iii)' with the inclusion of the noisy measurement $v^{\delta}.$

It is well known
that some regularization operator can be defined explicitly
for the objective functionals in the form of 
(\ref{obj_functional0}) with 
$\mathcal{J}(\cdot) = \frac{1}{2}\Vert \cdot\Vert^2,$ 
\textbf{\cite[Definition 3.1]{Engl96}}. This is not the case
anymore in the variational regularization strategy due to
the non-smoothness of the penalty term $\mathcal{J}$.
However, a regularization procedure definition 
can still be given fulfilling Hadamard`s principle.
Depending on the asymptotics of the regularization parameter
$\alpha,$ 
\begin{eqnarray}
\label{regularization_strategy}
\alpha(\delta , v^{\delta}) \rightarrow 0 \mbox{ and } \frac{\delta^2}{\alpha(\delta , v^{\delta})} \rightarrow 0
\mbox{, as } \delta \rightarrow 0,
\end{eqnarray}
regularization theory is concerned with the error estimation for the difference between the approximately regularized 
solution $u_{\alpha}^{\delta}$ and minimum norm solution
$u^{\dagger}$
of the inverse problem (\ref{inverse_problem}) as defined 
before.

Iterative regularization methods aim to produce stable approximate
solutions to the problem (\ref{problem0}) throughout some
iterative procedure, \textit{e.g.,} with the focus on minimizing
the discrepancy $\Vert T u - v^{\delta}\Vert$ by producing
a sequence of iterates $\hat{u}_{i}.$ Usually the iteration
is terminated at the iteration step $N$ by the choice of some stopping 
criterion, which is {\em a-posteriori} 
$N = N(v^{\delta} , \delta).$
This work concerns with the choice of the regularization parameter
not only in the continuous sense \textit{i.e.,} 
$\alpha = \alpha(v^{\delta} , \delta),$
but also within the iterative procedure \textit{i.e.,} 
$\alpha_{n} = \alpha_{n}(v^{\delta} , \delta).$
The scientific notion behind these notations can be reviewed
in \textbf{\cite[pp. 63]{Schuster12}}. 
The following definition is a quick adaptation of
\textbf{\cite[Definition 3.20]{Schuster12}} for our mathematical
setting and furthermore suits our iterative regularization scheme
which shall be introduced in Section \ref{iterative_convergence}.

Therefore, we will investigate the stability of the iterative
procedure as the number of the iterative steps tend to infinity.
In this work, the iterative regularization procedure 
involves proximal mappings.

\begin{definition}\textbf{[Proximal mapping]}
\label{def_prox}
Let $\mathcal{J} : \mathbb{R}^{N} \rightarrow \mathbb{R} \cup \{ \infty \}$
be a proper, convex, lower-semicontinious function.
Then $\mathrm{prox}_{\mathcal{J}}$ is defined as the unique minimizer
\begin{eqnarray}
\nonumber
\mathrm{prox}_{\mathcal{J}}(\tilde{u}) := 
\argmin_{u \in \mathbb{R}^{N}} \mathcal{J}(u) + 
\frac{1}{2} \Vert u - \tilde{u}\Vert^2.
\end{eqnarray}
\end{definition}

\subsection{Smoothness of the minimum norm solution under variational inequalities}
\label{data_smoothness}

Measuring the deviation of the regularized solution 
$u_{\alpha}^{\delta}$ 
from the minimum norm solution 
by \textit{a-priori} and \textit{a-posteriori}
strategies for the choice of the regularization parameter
in Banach spaces with the VSC has been widely studied
\textbf{\cite{BenningSchoenlib17},
\cite{Flemming18}, \cite{Grasmair10},   
\cite[Eq. (1.4)]{HofmannMathe12}, \cite[Section 4]{Kindermann16}
\cite[Theorem 2.60 - (g), Subsection 3.2.4]{Schuster12}}.
The objective is to bound the total error estimation
function defined, for some coefficient
$\Lambda \in \mathbb{R}_{+}$ depending on the functional 
properties of $\mathcal{J}$ and its domain, by
\begin{eqnarray}
E: & \mathcal{X} \times \mathcal{X} & \longrightarrow \mathbb{R}_{+}
\nonumber\\
\label{total_err_est}
& (u_{\alpha}^{\delta} , u^{\dagger}) & \longrightarrow
E(u_{\alpha}^{\delta} , u^{\dagger}) := 
\Lambda\Vert u_{\alpha}^{\delta} - u^{\dagger}\Vert.
\end{eqnarray} 
Different forms of the VSC have been considered for establishing
convergence and convergence rates results. A recent and concise
work on the verification of the VSCs in general terms has been
studied by Flemming \textit{et.al.,} 2018 
\textbf{\cite{Flemming18}}.

Our work does not focus on verification of the VSC. However,
by using fundamental functional analysis, it is still possible
to give mathematical motivation for the formulation of our VSC.
First, according to the Poincar\'{e}-Wirtinger inequality, see
\textbf{\cite[Theorem 3.1]{Bergounioux11}}, for some 
$u \in \mathrm{BV}(\Omega),$ there exists some constant
$C_{\Omega}$ such that
\begin{eqnarray}
\Vert\nabla u\Vert_1 \geq \frac{1}{C_{\Omega}} 
\Vert u - \mathrm{MV}[u]\vec{1}\Vert_1
\nonumber
\end{eqnarray}
holds. Addition to the fundamental inequality,
the early established result in 
\textbf{\cite[Eq. (3.43)]{Altuntac16}} and 
equivalence of the norm in the finite dimensional setting, 
lead to
\begin{eqnarray}
\Vert\nabla u\Vert_1 \geq \frac{1}{C_{\Omega}} 
\Vert u - \mathrm{MV}[u]\vec{1}\Vert_1 \geq 
\frac{1}{2}\Vert u\Vert_1
\geq \frac{1}{2}\Vert u\Vert_2 .
\nonumber
\end{eqnarray}
%
\noindent With that being stated, VSC could rather be formulated as a direct
estimator for the total error functional (\ref{total_err_est}).
\begin{assump}
\label{assump_conventional_variational_ineq}
\textbf{[Variational Source Condition]}
Let $T : \mathcal{X} \rightarrow \mathcal{Y}$ be linear, injective 
forward operator and $v^{\dagger} \in \mathrm{range}(T).$ 
There exists some constant 
$\sigma \equiv \sigma(C_{\Omega})\in (0 , 1]$
and a concave, monotonically increasing
index function $\Psi$ with $\Psi(0) = 0$ and 
$\Psi : [0 , \infty) \rightarrow [0 , \infty)$
such that for $q^{\dagger} \in \partial \mathcal{J}(u^{\dagger})$
the minimum norm solution $u^{\dagger}\in\mathrm{BV}(\Omega)$
satisfies
\begin{eqnarray}
\label{variational_ineq}
\sigma\Vert u-u^{\dagger}\Vert \leq 
\mathcal{J}(u) - \mathcal{J}(u^{\dagger})
+ \Psi\left( \Vert T u - T u^{\dagger}\Vert \right)  
\mbox{, for all } u \in \mathcal{X} .
\end{eqnarray} 
\end{assump}
Above, the estimation of the coefficient function $\sigma$
is an open question to be answered under the consideration of
$\mathcal{J} = TV.$

In some works, the Bregman distance itself is directly taken as
the estimator for the total error functional (\ref{total_err_est})
and VSC is formulated
as an upper bound for the Bregman distance, see \textit{e.g.}
\textbf{\cite{HofmannMathe12}}. The reader can find our further
convergence result based on this formulation of the VSC
in the Appendix \ref{VSC2}.

One straightforward conclusion of VSC follows from 
non-negativity of the total error functional (\ref{total_err_est}),
\begin{eqnarray}
\label{J-diff1}
\mathcal{J}(u^{\dagger}) - \mathcal{J}(u) \leq 
\Psi\left( \Vert T u - T u^{\dagger}\Vert\right),
\end{eqnarray}
for all $u \in \mathcal{X}.$ We furthermore state one fundamental idendity
that is valid for the concave functions, see 
\textbf{\cite[Proposition 1]{HofmannMathe12},
\cite[Eq. (4.37)]{Schuster12}},
\begin{eqnarray}
\label{concave_id}
\Psi(K \delta) \leq K\Psi(\delta), \mbox{ for } K \geq 1.
\end{eqnarray}

We will develop convergence rates results from two different
aspects. Firstly, the convergence of the non-iterative
regularized solution $u_{\alpha}^{\delta}$
of the problem (\ref{problem0})
towards the minimum norm solution $u^{\dagger}.$ 
We conduct this analysis in the continuous setting. 
Secondly, we study the convergence of the
iteratively regularized solution 
$\hat{u}_{\alpha_{i}}^{\delta},$
where $i = 1, 2, \cdots$ indicate the iteration step,
produced by some primal-dual algorithm still towards 
the minimum norm solution $u^{\dagger}.$
In both investigations, the convergence rates  
will be expressed in terms of the index
function $\Psi$ and be achieved by \textit{a-posteriori} choice 
of the regularization parameter $\alpha$ , see Section
\ref{choice_of_regpar} for the details.

%


\subsection{Choice of the Regularization Parameter: Morozov's Discrepancy Principle}
\label{choice_of_regpar}

We are concerned with asymptotic properties of the
regularization parameter $\alpha$ for the Tikhonov-regularized 
solution obtained by Morozov's discrepancy principle (MDP).
MDP serves as an \textit{a posteriori}
parameter choice rule for the Tikhonov type objective functionals  
(\ref{obj_functional2}) and has certain impact 
on stabilizing the total error functional 
$E$ having the assumed relation (\ref{total_error_bregman}).
As has been introduced in 
\textbf{\cite[Theorem 3.10]{AnzengruberRamlau10}} 
and \textbf{\cite{AnzengruberRamlau11}},
we use the following set notations
in the theorem formulations that are necessary
to establish the error estimation between the regularized solution
$u_{\alpha}^{\delta}$ for the problem (\ref{problem0})
and the minimum norm solution $u^{\dagger},$
\begin{eqnarray}
\label{MDP1}
\overline{S} & := & \left\{ \alpha : 
\Vert T u_{\alpha}^{\delta} - v^{\delta}\Vert 
\leq \overline{\tau}\delta \mbox{ for } 
u_{\alpha}^{\delta} = \argmin_{u \in \mathcal{X}} 
\{ F_{\alpha}(u , v^{\delta})\} \right\} ,
\\
\label{MDP2}
\underline{S} & := & \left\{ \alpha : \underline{\tau}\delta \leq 
\Vert T u_{\alpha}^{\delta} - v^{\delta}\Vert 
\mbox{ for } 
u_{\alpha}^{\delta} = \argmin_{u \in \mathcal{X}} 
\{ F_{\alpha}(u , v^{\delta})\} \right\} ,
\end{eqnarray}
where the discrepancy set radii 
$1 < \underline{\tau} \leq \overline{\tau} < \infty$ are fixed.
Analogously, also as well known from
\textbf{\cite[Eq. (4.57) and (4.58)]{Engl96}} and
\textbf{\cite[Definition 2.3]{Kirsch11}},
we are interested in such a regularization parameter 
$\alpha(\delta , v^{\delta}),$ 
with some fixed discrepancy set radii 
$1 < \underline{\tau} \leq \overline{\tau} < \infty ,$ that
\begin{eqnarray}
\label{discrepancy_pr_definition}
\alpha(\delta , v^{\delta}) \in \{ \alpha > 0 
\mbox{ }\vert \mbox{ }
\underline{\tau}\delta \leq 
\Vert T u_{\alpha}^{\delta} - v^{\delta}\Vert 
\leq \overline{\tau}\delta \} 
= \overline{S} \cap \underline{S} \mbox{ for the given } 
(\delta , v^{\delta}).
\end{eqnarray}
It is also the immediate consequences of MDP that 
the following estimations
\begin{eqnarray}
\label{consequence_MDP}
\Vert T u_{\alpha}^{\delta} - T u^{\dagger}\Vert \leq 
(\overline{\tau} + 1)\delta,
\\
\label{consequence_MDP2}
(\underline{\tau} - 1)\delta \leq 
\Vert T u_{\alpha}^{\delta} - T u^{\dagger}\Vert,
\end{eqnarray}
hold true. 


\subsection{Iterative regularization}
\label{subsection_iter_reg}
By an iterative procedure involving some iteration operator $R_{I},$
\cite[Ch. 6]{BenningBurger17},
we aim to construct some approximation to the given inverse
ill-posed problem (\ref{inverse_problem})
\begin{eqnarray}
\label{iterative_proc}
u_{i} = R_{I}(v^{\delta} , w_{i-1} , \Gamma),
\end{eqnarray}
where $w^{i-1}$ is the collection of dual variables used 
during $i-1$ iteration steps, and $\Gamma$ is the 
auxiliary parameters such as step-size, relaxation parameter, 
regularization parameter.

In the iterative regularization procedures, discrepancy principles act as
the stopping rules for the corresponding algorithms, 
\textbf{\cite[Section 6]{BenningBurger17}}.

\begin{definition}{[Morozov`s Discrepancy Principle (MDP), 
\cite[Def. 6.1]{BenningBurger17}]}

Given deterministic noise model 
$\Vert v^{\dagger}-v^{\delta}\Vert \leq \delta,$ if we choose
$\tau > 1$ and $i^{\ast} = i^{\ast}(\delta , v^{\delta})$ such that
\begin{eqnarray}
\Vert T u_{i^{\ast}} - v^{\delta}\Vert \leq 
\tau \delta < \Vert T u_{i} - v^{\delta}\Vert
\end{eqnarray}
is satisfied for $u_{i^{\ast}} = 
R_{I}(v^{\delta} , w_{i^{\ast}-1} , \vec{\alpha})$ and 
$u_{i} = R_{I}(v^{\delta} , w_{i-1} , \vec{\alpha})$ for all
$i < i^{\ast},$ then $u_{i^{\ast}}$ is said to satisfy {\em Morozov`s 
discrepancy principle}.
\end{definition}
Following up MDP, some immediate consequences 
can be given below,
\begin{eqnarray}
\label{consequence_MDP}
\Vert T u_{i^{\ast}}^{\delta} - T u^{\dagger}\Vert 
\leq (\tau + 1)\delta ,
\end{eqnarray}
likewise,
\begin{eqnarray}
\tau\delta \leq \Vert T u_{i} - v^{\delta}\Vert & \Rightarrow & 
\tau\delta \leq \Vert T u_{i} - T u^{\dagger}\Vert + \delta
\nonumber\\ 
& \Rightarrow & 
(\tau - 1)\delta \leq \Vert T u_{i} - T u^{\dagger}\Vert
\nonumber\\ 
& \Rightarrow & 
(\tau - 1)^2\delta^2 \leq \Vert T u_{i} - T u^{\dagger}\Vert^2
\nonumber\\ 
& \Rightarrow &
-\Vert T u_{i} - T u^{\dagger}\Vert^2 \leq -(\tau - 1)^2\delta^2 .
\label{consequence_MDP2}
\end{eqnarray}


\section{Subgradient Characterization of the Regularized Solution $u_{\alpha}^{\delta}$: Primal-Dual Splitting}
\label{section_subdiff_char}

Still in the continuous setting, the subgradient characterization
of the regularized solution $u_{\alpha}^{\delta}$ of the problem
(\ref{problem0}) is to be studied. Our iterative regularization algorithm
will arise from characterization of the primal and the dual
solutions. This characterization will form some coupled system
which makes the primal solution depend on the dual one, and vice versa.
Before moving on to the characterization, let us introduce some
more notations.
Since $u_{\alpha}^{\delta}$ is the minimizer of the 
objective functional (\ref{obj_functional2}), then 
by the first order optimality condition
\begin{eqnarray}
\nonumber
0 \in \partial F_{\alpha}(u_{\alpha}^{\delta} , v^{\delta}),
\end{eqnarray}
which implies,
\begin{eqnarray}
\label{char_opt_cond}
0 = \frac{1}{\alpha}T^{T}(T u_{\alpha}^{\delta} - v^{\delta})
+ \partial \mathcal{J}(u_{\alpha}^{\delta}) - 
\partial \mathcal{J}(u^{0}) + \hat{z} 
\mbox{, where }
\hat{z} \in \partial h(u_{\alpha}^{\delta}).
\end{eqnarray}
Furthermore, recall the settings in (\ref{composite_TV}) 
and (\ref{subdiff_TV}) to represent (\ref{char_opt_cond})
in the following form
\begin{eqnarray}
\nonumber
0 = \frac{1}{\alpha} T^{T}(T u_{\alpha}^{\delta} - v^{\delta})
+ D^{T} \hat{w}_{\alpha}^{\delta} - D^{T} \hat{w}^{0} + \hat{z} ,
\end{eqnarray}
where $\hat{w}_{\alpha}^{\delta} \in \partial g(Du_{\alpha}^{\delta})$
and likewise
$\hat{w}^{0} \in \partial g(Du^{0}),$ and we recall that
$g(\cdot) = \Vert\cdot\Vert.$ Then, the subdifferential characterization
of the regularized solution is formulated below.

\begin{theorem}
\label{thrm_subgrad_char}
The regularized minimizer $u_{\alpha}^{\delta}$ of the 
objective functional (\ref{obj_functional2}) is characterized by
\begin{eqnarray}
\label{eq_subgrad_char}
\left\{ \begin{array}{rcl}
u_{\alpha}^{\delta} &=& \mathrm{prox}_{\mu h}\left[u_{\alpha}^{\delta}  
- \mu \left( \frac{1}{\alpha}T^{T}(T u_{\alpha}^{\delta} - v^{\delta})
+ D^{T}(\hat{w}_{\alpha}^{\delta} - \hat{w}^{0} ) \right) \right]
\\
\hat{w}_{\alpha}^{\delta} &=& \mathrm{prox}_{\nu g^{\ast}} 
\left( \hat{w}_{\alpha}^{\delta} + \nu Du_{\alpha}^{\delta} \right),
\end{array}\right.
\end{eqnarray}
with $\hat{w}^{0} = \partial \Vert D u^{0}\Vert_{1}.$
\end{theorem}

\begin{proof}
For some $\mu > 0,$ the primal solution characterization 
follows from the inclusions below
\begin{eqnarray}
\hat{z} \in \partial h(u_{\alpha}^{\delta}) & \Leftrightarrow & 
0 \in - \mu\hat{z} + \mu\partial h(u_{\alpha}^{\delta})
\nonumber\\
& \Leftrightarrow & 0 \in \tilde{u} - u_{\alpha}^{\delta} - 
\mu\hat{z} + \partial \mu h(\tilde{u}) 
\mbox{ at } \tilde{u} = u_{\alpha}^{\delta}
\nonumber\\
& \Leftrightarrow & 0 \in \partial_{\tilde{u}} \left(\frac{1}{2} 
\left\Vert \tilde{u} - \left( u_{\alpha}^{\delta} + \mu\hat{z} \right) \right\Vert^2
+ \mu h(\tilde{u}) \right)\mbox{ at } \tilde{u} = u_{\alpha}^{\delta}
\nonumber\\
& \Leftrightarrow &
u_{\alpha}^{\delta} \in \argmin_{\tilde{u}} \frac{1}{2} 
\left\Vert \tilde{u} - \left( u_{\alpha}^{\delta} + \mu\hat{z} \right) \right\Vert^2
+ \mu h(\tilde{u})
\nonumber\\
& \Leftrightarrow &
u_{\alpha}^{\delta} = \mathrm{prox}_{\mu h}(u_{\alpha}^{\delta} + \mu \hat{z}).
\end{eqnarray}
Regarding the dual characterization, 
again for some $\nu > 0,$ consider the following inclusions,
\begin{eqnarray}
\hat{w}_{\alpha}^{\delta} \in \partial g(D u_{\alpha}^{\delta}) & \Leftrightarrow & D u_{\alpha}^{\delta} \in \partial g^{\ast}(\hat{w}_{\alpha}^{\delta})
\nonumber\\
& \Leftrightarrow & 0 \in -\nu D u_{\alpha}^{\delta} + \nu
\partial g^{\ast}(\hat{w}_{\alpha}^{\delta})
\nonumber\\
& \Leftrightarrow & 0 \in \xi - \hat{w}_{\alpha}^{\delta} - \nu D u_{\alpha}^{\delta} + \nu\partial g^{\ast}(\xi) 
\mbox{ at } \xi = \hat{w}_{\alpha}^{\delta}
\nonumber\\
& \Leftrightarrow & 0 \in \xi - (\hat{w}_{\alpha}^{\delta} + \nu D u_{\alpha}^{\delta}) + \nu \partial g^{\ast}(\xi)
\mbox{ at } \xi = \hat{w}_{\alpha}^{\delta}
\nonumber\\
& \Leftrightarrow & 0\in \partial_{\xi}\left(
\frac{1}{2}\left\Vert \xi - 
(\hat{w}_{\alpha}^{\delta} + \nu D u_{\alpha}^{\delta})\right\Vert^2 +  
\nu g^{\ast}(\xi)
\right)\mbox{ at } \xi = \hat{w}_{\alpha}^{\delta}
\nonumber\\
& \Leftrightarrow & \hat{w}_{\alpha}^{\delta} \in 
\argmin_{\xi}
\left\Vert \xi - (\hat{w}_{\alpha}^{\delta} + 
\nu D u_{\alpha}^{\delta}) \right\Vert^2 +  \nu g^{\ast}(\xi)
\nonumber\\
& \Leftrightarrow & \hat{w}_{\alpha}^{\delta} = 
\mathrm{prox}_{\nu g^{\ast}} 
\left(\hat{w}_{\alpha}^{\delta} + \nu Du_{\alpha}^{\delta} \right).
\end{eqnarray}
The asserted form of the characterization is obtained solving
for $\hat{z}$ according to (\ref{char_opt_cond}).
\end{proof}

\section{Continuous Stability Analysis: Convergence of the Regularized Minimizer to the Minimum Norm Solution}
\label{stab_analysis}
We will investigate the conditions under which the
regularized solution $u_{\alpha}^{\delta}$ of the 
problem (\ref{problem0}) converges to the 
minimum norm solution $u^{\dagger}.$
One of the conditions to be presented is also the choice
of the initial guess $u^{0}.$

\subsection{Bounds for the regularization parameter}
\label{lower_bound_reg_par}
As has been motivated above,
our choice of regularization parameter must fulfill (\ref{regularization_strategy}).
Moving on fom here, 
it is possible to obtain quantitative upper bounds 
for the Bregman distance 
$D_{\mathcal{J}}$, or for the total error value functional $E,$ see (\ref{total_error_bregman}) in the Appendix. Following 
(\ref{regularization_strategy}), the singularity of the
quotient $\frac{\delta^2}{\alpha}$ will be controlled as
$\alpha \rightarrow 0$ whilst
$\delta \rightarrow 0.$ 
It has been observed in the literature,
\textbf{\cite[Corollary 4.2]{AnzengruberRamlau10}, 
\cite[Lemma 2.8]{AnzengruberRamlau11}, 
\cite[Eq. (2.17)]{AnzengruberHofmannMathe14}, 
\cite[Theorem 4.4]{Grasmair13} 
\cite[Lemma 1]{HofmannMathe12}},
that this controllability
is possible when the choice of the regularization
parameter obeys MDP.
As a result of this choice and of the inclusion of the VSC, 
the quantitave estimations for the Bregman distance
depend on the discrepancy set radii and the coefficient in the VSC.

We will observe that mathematical properties 
of the index function $\Psi$ that are of monotonicity 
and concavity together with
the choice of the regularization parameter
$\alpha = \alpha(\delta , v^{\delta}) \in \underline{S}$ 
enables one to control the indefinite limit case
$\frac{\delta^2}{\alpha}$ as $\alpha \rightarrow 0$
whilst $\delta \rightarrow 0.$ When one considers 
seeking the minimizer of the objective
functional (\ref{obj_functional0}), the control
over $\frac{\delta^2}{\alpha}$ provides some lower bound
for the regularization parameter
\begin{eqnarray}
\label{lower_bound_regpar}
\nonumber
\alpha(\delta , v^{\delta}) \geq 
\frac{1}{4}\frac{\underline{\tau}^2 - 1}{\underline{\tau}^2 + 1}
\frac{\delta^2}{\Psi((\underline{\tau} - 1)\delta)},
\end{eqnarray}
see \textbf{\cite[Eq (3.29)]{Altuntac16}} and 
\textbf{\cite[Corollary 2]{HofmannMathe12}}.
However, this bound is only valid for the functional in the form of
(\ref{obj_functional0}). Therefore, a new lower bound for 
the regularization
parameter that is in line with our objective functional 
$F_{\alpha}$ by
(\ref{obj_functional2}) must be estimated. 

\begin{lemma}
\label{lemma_regpar_lowerbnd}
Let $u_{\alpha}^{\delta}\in\Omega$ be the regularized minimizer 
for the objective functional $F_{\alpha}$ in (\ref{obj_functional2}) 
with the initial guess
$u^{0} \in \mathbb{R}^{N}$ that is some constant, and 
let the minimum norm solution $u^{\dagger}$ satisfy the VSC
(\ref{variational_ineq}). Given the
regularization parameter that is chosen {\em a-posteriori} 
$\alpha = \alpha(\delta , v^{\delta}) \in \overline{S} \cap \underline{S},$ the singularity of $\frac{\delta^2}{\alpha}$
as $\alpha \rightarrow 0$ whilst $\delta \rightarrow 0$
is controlled by 
\begin{eqnarray}
\label{lower_alpha}
\frac{\delta^2}{2\alpha}\left(\underline{\tau}-1\right) 
\leq \Psi(\delta).
\end{eqnarray}
\end{lemma}

\begin{proof}
In analogous with the subgradient characterization, 
we assign $\hat{w}^{0} \in \partial g(D u^{0})$
with $g(\cdot) = \Vert \cdot\Vert_{1}.$ Thus,
\begin{eqnarray}
\hat{w}^{0} := \left\{ \begin{array}{rcl}
\{ -1 \} & \mbox{if} & [D u^{0}] < 0 \\ 
{[}-1,1{]} & \mbox{if} & [D u^{0}] = 0 \\
\{ 1 \} & \mbox{if} & [D u^{0}] > 0 .
\end{array}\right.
\nonumber
\end{eqnarray}
Since the initial guess $u^{0}$ is set to some constant, 
thus we consider $w_{0} = 0.$ Furthermore, both the regularized
minimizer and the minimum norm solutions are in the constraint
domain. So, $h(u_{\alpha}^{\delta}) = h(u^{\dagger}) = 0.$
The fact that $u_{\alpha}^{\delta} = 
\argmin_{u \in \mathcal{X}}F_{\alpha}(u , v^{\delta})$
for all $u\in\mathcal{X}$ provides us,
\begin{eqnarray}
\nonumber
F_{\alpha}(u_{\alpha}^{\delta} , v^{\delta})
\leq F_{\alpha}(u^{\dagger} , v^{\delta}),
\end{eqnarray}
which is explicitly,
\begin{eqnarray}
\nonumber
\frac{1}{2}\Vert T u_{\alpha}^{\delta} - v^{\delta}\Vert^2 +
\alpha D_{\mathcal{J}}(u_{\alpha}^{\delta} , u^{0})
\leq \frac{\delta^2}{2} + 
\alpha D_{\mathcal{J}}(u^{\dagger} , u^{0}) \rangle .
\end{eqnarray}
On the right hand side, we have straightforwardly used
$\Vert T u^{\dagger} - v^{\delta}\Vert^2 \leq \delta^{2}$ since
$v^{\delta} \in \mathcal{B}_{\delta}(v^{\dagger})$ for
$T u^{\dagger} = v^{\dagger}.$
Now, after multiplying both sides by $\frac{1}{\alpha},$ 
the choice of the regularization
parameter and the definition of the Bregman distance
together with (\ref{J-diff1}) yield,
\begin{eqnarray}
(\underline{\tau}^2 - 1)\frac{\delta^2}{2\alpha}& \leq &
D_{\mathcal{J}}(u^{\dagger} , u^{0}) - 
D_{\mathcal{J}}(u_{\alpha}^{\delta} , u^{0})
\nonumber\\
& = & \mathcal{J}(u^{\dagger}) - \mathcal{J}(u_{\alpha}^{\delta}) + 
\langle D^{\ast}w^{0} , u^{0} - u^{\dagger} \rangle +
\langle D^{\ast}w^{0} , u_{\alpha}^{\delta} - u^{0} \rangle
\nonumber\\
& \stackrel{(\ref{J-diff1})}{\leq} & 
\Psi(\Vert T u^{\dagger} - T u_{\alpha}^{\delta}\Vert)
+ \langle D^{\ast}w^{0} , u_{\alpha}^{\delta} - 
u^{\dagger}\rangle = \Psi(\Vert T u^{\dagger} - T u_{\alpha}^{\delta}\Vert).
\nonumber
\end{eqnarray}
From here on, monotonicty and concavity of the index
function $\Psi$ imply
\begin{eqnarray}
\nonumber
\Psi(\Vert T u^{\dagger} - T u_{\alpha}^{\delta}\Vert)
\stackrel{(\ref{consequence_MDP})}{\leq}
 \Psi((\overline{\tau} + 1)\delta)
\stackrel{(\ref{concave_id})}{\leq}(\overline{\tau}+1)\Psi(\delta).
\end{eqnarray}
Eventually, for some fixed discrepancy radii
$1 < \underline{\tau} \leq \overline{\tau} < \infty,$
\begin{eqnarray}
\label{delta2alpha_tradeoff}
\frac{\delta^2}{2\alpha}\leq 
\left(\frac{1}{\underline{\tau}-1}\right)\Psi(\delta).
\end{eqnarray}
\end{proof}

\begin{lemma}
\label{lemma_J_diff}
Under the same assumptions formulated in Lemma 
\ref{lemma_regpar_lowerbnd}, the following estimation
for the $\mathcal{J}$ difference holds
\begin{displaymath}
\mathcal{J}(u_{\alpha}^{\delta}) - \mathcal{J}(u^{\dagger}) = 
\mathcal{O}(\Psi(\delta)).
\end{displaymath}
\end{lemma}

\begin{proof}
Likewise in the previous proof, the regularized minimizer 
is global minimum for the functional $F_{\alpha},$
where $\alpha(\delta , v^{\delta}) \in \underline{S} \cap \overline{S},$
\begin{eqnarray}
\nonumber
\frac{1}{2}\Vert T u_{\alpha}^{\delta} - v^{\delta}\Vert^2 +
\alpha\mathcal{J}(u_{\alpha}^{\delta}) - \alpha\mathcal{J}(u^{0}) 
& - &\alpha\langle D^{\ast} w^{0}, 
u_{\alpha}^{\delta} - u^{0} \rangle
\leq 
\nonumber\\
& \frac{\delta^2}{2} & + \alpha\mathcal{J}(u^{\dagger}) - 
\alpha\mathcal{J}(u^{0})
- \alpha\langle D^{\ast} w^{0}, u^{\dagger} - u^{0} \rangle .
\end{eqnarray}
Here, the inner products on the both sides drop because of
the initial guess $u^{0} \equiv \mbox{constant},$ also
the both sides are multiplied by $\frac{1}{\alpha}.$ 
Thus, this yields, for 
$1 < \underline{\tau} \leq \overline{\tau} < \infty,$
\begin{eqnarray}
\label{J-diff2}
\mathcal{J}(u_{\alpha}^{\delta}) - \mathcal{J}(u^{\dagger}) 
\leq \frac{\delta^{2}}{2\alpha}
\stackrel{(\ref{delta2alpha_tradeoff})}{\leq} 
\left(\frac{1}{\underline{\tau}-1}\right)\Psi(\delta).
\end{eqnarray}
\end{proof}

\begin{cor}
Both (\ref{J-diff1}) and (\ref{J-diff2}) imply the following, 
cf. \textbf{\cite[(2) of Theorem 1]{HofmannMathe12}},
\begin{displaymath}
\vert\mathcal{J}(u_{\alpha}^{\delta}) - \mathcal{J}(u^{\dagger}) 
\vert = \mathcal{O}(\Psi(\delta)).
\end{displaymath}
\end{cor}

The theoretical preparation established thus far reveals
one stable upper bound over the total error estimation below.
\begin{theorem}
\label{theorem_total_err}
In the light of the assumptions of lemmata \ref{lemma_regpar_lowerbnd}
and \ref{lemma_J_diff}, the following estimation holds,
\begin{displaymath}
E(u_{\alpha}^{\delta} , u^{\dagger}) = \mathcal{O}(\Psi(\delta)).
\end{displaymath}
\end{theorem}

\begin{proof}
Since the minimum norm solution 
$u^{\dagger} \in \mathrm{BV}(\Omega)$ is assumed to satisfy
the VSC in (\ref{variational_ineq}), then 
simply by following the similar arguments above
for the concave, monotonically increasing index function 
$\Psi :[0 , \infty) \rightarrow [0 , \infty),$
\begin{eqnarray}
E(u_{\alpha}^{\delta} , u^{\dagger}) & \leq & 
\mathcal{J}(u_{\alpha}^{\delta}) - \mathcal{J}(u^{\dagger}) + 
\Psi(\Vert T u_{\alpha}^{\delta} - T u^{\dagger}\Vert)
\nonumber\\
& \leq & \left(\frac{1}{\underline{\tau}-1}\right)\Psi(\delta)
+ (\overline{\tau}+1)\Psi(\delta).
\end{eqnarray}
\end{proof}
In the Appendix \ref{VSC2}, an analagous result is formulated
with a more involved form of the VSC.

Theorem \ref{theorem_total_err} has been in fact asserted owing to the
lower bound for the regularization parameter given by 
Lemma \ref{lemma_regpar_lowerbnd}. MDP also provides some upper bound
for the regularization parameter which arises another quantitative
stability analysis.

\begin{lemma}
\label{misfit_upper_bnd}
Under the Assumption \ref{assump_conventional_variational_ineq} and
the initial guess $u^{0}$ of the objective functional $F_{\alpha}$
in (\ref{obj_functional2}) to be some constant
for $\alpha > 0$ and the regularized solution $u_{\alpha}^{\delta}$
of the problem (\ref{problem0}), we have
\begin{eqnarray}
\frac{1}{2}\Vert T u_{\alpha}^{\delta} - T u^{\dagger}\Vert^2 
\leq 2\delta^2 + 
2\alpha\Psi\left(\Vert T u_{\alpha}^{\delta} - T u^{\dagger}\Vert \right).
\end{eqnarray}
\end{lemma}

\begin{proof}
Again from $u_{\alpha}^{\delta} \in \argmin_{u} F_{\alpha}(u)$ and
$u^{0}$ to be some constant, then one immediately obtains
\begin{eqnarray}
\mathcal{J}(u_{\alpha}^{\delta}) - \mathcal{J}(u^{\dagger}) \leq \frac{\delta^{2}}{2\alpha} 
- \frac{1}{2}\Vert T u_{\alpha}^{\delta} - v^{\delta}\Vert^2
\end{eqnarray}
This is replaced by the $\mathcal{J}$ difference in (\ref{variational_ineq})
in the following way
\begin{eqnarray}
\label{J-diff_replace}
0 \leq \frac{\delta^2}{2\alpha} - 
\frac{1}{2\alpha}\Vert T u_{\alpha}^{\delta} - v^{\delta}\Vert^2 +
\Psi\left(\Vert T u_{\alpha}^{\delta} - T u^{\dagger}\Vert \right).
\end{eqnarray}
On the other hand, since $v^{\delta} \in \mathcal{B}_{\delta}(v^{\dagger}),$
\begin{eqnarray}
\nonumber
\Vert T u_{\alpha}^{\delta} - T u^{\dagger}\Vert \leq 
\Vert T u_{\alpha}^{\delta} - v^{\delta}\Vert + \delta
\end{eqnarray}
from which, the following can be obtained by using 
$ab \leq 2(a^2 + b^2)$ for $a,b \geq 0,$ 
\begin{eqnarray}
\nonumber
\frac{1}{2\alpha}\delta^2 - \frac{1}{4\alpha}
\Vert T u_{\alpha}^{\delta} - T u^{\dagger}\Vert^2 \geq 
- \frac{1}{2\alpha}\Vert T u_{\alpha}^{\delta} - v^{\delta}\Vert^2.
\end{eqnarray} 
Inserting this into (\ref{J-diff_replace}) brings
\begin{eqnarray}
0 \leq \frac{\delta^2}{\alpha} - \frac{1}{4\alpha}
\Vert T u_{\alpha}^{\delta} - T u^{\dagger}\Vert^2 + 
\Psi\left( \Vert T u_{\alpha}^{\delta} - T u^{\dagger}\Vert \right),
\end{eqnarray}
which yields the assertion yields after some algebraic arrangement.
\end{proof}
It is expected to find some upper bound for the regularization parameter
$\alpha \leq \overline{\alpha}.$ Given specific value of the parameter
$\alpha >0,$ it is possible to bound the misfit term 
$\Vert T u_{\alpha}^{\delta} - v^{\delta}\Vert.$
To this end, we will make use of the bound in 
Lemma \ref{misfit_upper_bnd}. Following the literature,
\textbf{\cite[Eq (3.2)]{HofmannMathe12}, \cite[pp. 59]{Schuster12}},
we also introduce the following function
as the upper bound $\overline{\alpha}$ for the regularization
parameter $\alpha,$
\begin{eqnarray}
\label{regpar_upper_bnd}
\overline{\alpha} = \Phi(t) := \frac{t^2}{\Psi(t)}, \mbox{ for } t > 0,
\end{eqnarray}
where $\Psi$ is the index function in (\ref{variational_ineq}).

\begin{lemma}
\label{misfit_upper_bnd2}
Suppose that the minimum norm solution $u^{\dagger}$ satisfies
Assumption \ref{assump_conventional_variational_ineq} for some concave and
monotonically increasing function $\Psi.$ 
Let the upper bound $\overline{\alpha}$ for the regularization parameter 
$\alpha(v^{\delta} , \delta) \in \overline{S}\cap\underline{S}$
be given as
\begin{displaymath}
\overline{\alpha} := \Phi(\delta).
\end{displaymath}
Then, for any $\alpha \leq \overline{\alpha}$ and the regularized 
solution $u_{\alpha}^{\delta}$ of the problem (\ref{problem0}), 
we have
\begin{eqnarray}
\label{est_misfit_upper_bnd2}
\frac{1}{4}\Vert T u_{\alpha}^{\delta} - T u^{\dagger}\Vert \leq 
\frac{\overline{\tau}+2}{\underline{\tau}-1}\delta .
\end{eqnarray}
\end{lemma}

\begin{proof}
The proof starts with the assertion of Lemma \ref{misfit_upper_bnd} 
in the following form,
\begin{eqnarray}
\frac{1}{4}\Vert T u_{\alpha}^{\delta} - T u^{\dagger}\Vert^{2} 
& \leq & 
\delta^2 + \alpha\Psi\left(\Vert T u_{\alpha}^{\delta} - T u^{\dagger}\Vert
\right)
\nonumber\\ & \leq &
\delta^2 + \overline{\alpha}
\Psi\left( \Vert T u_{\alpha}^{\delta} - T u^{\dagger}\Vert\right)
\nonumber\\ & \stackrel{(\ref{regpar_upper_bnd})}{=} &
\delta^2 + 
\delta^2\frac{\Psi\left( \Vert T u_{\alpha}^{\delta} - T u^{\dagger}\Vert
\right)}{\Psi(\delta)}
\nonumber\\ & \stackrel{(\ref{MDP2})}{\leq} &
\delta^2 \frac{1}{(\underline{\tau} - 1)\delta}
\Vert Tu_{\alpha}^{\delta} - Tu^{\dagger}\Vert + 
\delta^2\frac{1}{\Psi(\delta)}
\Psi\left( (\underline{\tau} - 1)\delta
\frac{\Vert T u_{\alpha}^{\delta} - T u^{\dagger}\Vert}
{(\underline{\tau} - 1)\delta} \right)
\nonumber\\ & \stackrel{(\ref{concave_id})}{\leq} &
\delta \frac{1}{(\underline{\tau} - 1)}
\Vert Tu_{\alpha}^{\delta} - Tu^{\dagger}\Vert +
\delta \frac{1}{\Psi(\delta)}\frac{1}{(\underline{\tau} - 1)}
\Vert T u_{\alpha}^{\delta} - T u^{\dagger}\Vert
\Psi\left((\underline{\tau} - 1)\delta \right)
\nonumber\\ &\leq &
\delta \frac{1}{(\underline{\tau} - 1)}
\Vert Tu_{\alpha}^{\delta} - Tu^{\dagger}\Vert +
\delta \frac{1}{\Psi(\delta)}\frac{1}{(\underline{\tau} - 1)}
\Vert T u_{\alpha}^{\delta} - T u^{\dagger}\Vert
\Psi\left((\overline{\tau} + 1)\delta \right)
\nonumber\\ & \stackrel{(\ref{concave_id})}{\leq} &
\delta \frac{1}{(\underline{\tau} - 1)}
\Vert Tu_{\alpha}^{\delta} - Tu^{\dagger}\Vert + 
\delta \frac{1}{(\underline{\tau} - 1)}
\Vert T u_{\alpha}^{\delta} - T u^{\dagger}\Vert
\left(\overline{\tau} + 1\right).
\end{eqnarray}
\end{proof}

\begin{theorem}
Under the same assumptions in Lemma \ref{misfit_upper_bnd2}, 
the total error estimation that is 
between the regularized solution $u_{\alpha}^{\delta}$
and the minimum norm solution $u^{\dagger}$ attains the bound
\begin{displaymath}
E(u_{\alpha}^{\delta} , u^{\dagger}) = 
\mathcal{O}\left(\frac{\delta^2}{\alpha}\right).
\end{displaymath} 
\end{theorem}

\begin{proof}
For the \textit{a-posteriori} choice of the regularization
parameter $\alpha \in \overline{S}\cap\underline{S}$ and some 
$1 < \underline{\tau}\leq\overline{\tau}<\infty ,$ following up the
(\ref{variational_ineq}),
\begin{eqnarray}
\frac{\sigma}{2}E(u_{\alpha}^{\delta},u^{\dagger}) 
&\leq & \mathcal{J}(u_{\alpha}^{\delta}) - 
\mathcal{J}(u^{\dagger}) + 
\Psi\left(\Vert T u_{\alpha}^{\delta} - T u^{\dagger}\Vert\right)
\nonumber\\ & \stackrel{(\ref{est_misfit_upper_bnd2})}{\leq} &
\mathcal{J}(u_{\alpha}^{\delta}) - \mathcal{J}(u^{\dagger}) + 
\Psi\left(\frac{\overline{\tau}+2}{\underline{\tau}-1}\delta\right)
\nonumber\\ &\stackrel{(\ref{J-diff2})\&(\ref{concave_id})}{\leq}&
\frac{1}{(\underline{\tau} - 1)}\Psi(\delta) + \left(\overline{\tau}+2\right)
\Psi\left( \frac{1}{\underline{\tau}-1}\delta\right)
\nonumber\\ & \leq &
\frac{1}{(\underline{\tau} - 1)}\Psi(\delta) + \left(\overline{\tau}+2\right)
\Psi\left( \frac{\overline{\tau}+1}{\underline{\tau}-1}\delta\right)
\nonumber\\ & \stackrel{(\ref{concave_id})}{\leq} &
\left(\frac{1}{\underline{\tau} - 1} + \left(\overline{\tau}+2\right)
\frac{\overline{\tau}+1}{\underline{\tau}-1}\right)\Psi(\delta)
\nonumber
\end{eqnarray}
which implies, due to
$\alpha \leq \overline{\alpha} := \frac{\delta^2}{\Psi(\delta)},$
\begin{eqnarray}
\label{total_err_regpar_upper_bnd}
E(u_{\alpha}^{\delta},u^{\dagger}) \leq \frac{2}{\sigma}
C(\underline{\tau},\overline{\tau})\frac{\delta^2}{\alpha},
\end{eqnarray}
where,
\begin{eqnarray}
C(\underline{\tau},\overline{\tau}) := 
\left( \frac{1}{\underline{\tau}-1} + \left(\overline{\tau}+2\right)
\frac{\overline{\tau}+1}{\underline{\tau}-1} \right).
\end{eqnarray}
\end{proof}

%
%

\section{Primal Dual Algorithm With Convex Extrapolation Step}
\label{primal-dual}


Our objective is to find an iteratively regularized approximation
$\hat{u}_{\alpha_{i+1}}^{\delta}$ of the regularized
solution for the problem (\ref{problem0}). To this end,
we will develop some proximal-gradient algorithm
evolving from the well-known formulation 
\textbf{\cite{BauschkeCombettes11, BeckTeboulle09, CombettesPesquet11}}, 
and the references therein,
\begin{eqnarray}
\label{prox_alg}
\hat{u}_{\alpha_{i+1}}^{\delta} = 
\mathrm{prox}_{\mu h}\left( \hat{u}_{\alpha_{i}}^{\delta} - 
\mu \frac{1}{\alpha_{i}} T^{T}
(T\hat{u}_{\alpha_{i}}^{\delta} - v^{\delta}) \right),
\end{eqnarray}
with an appropriately chosen step-length $\mu > 0.$
The stability of the iteratively
regularized approximation 
$\hat{u}_{\alpha_{i+1}}^{\delta}$ of $u^{\dagger}$ in 
(\ref{J-min_est}) will be studied in the Hadamard sense. 

Our algorithm consists of two nested loops. The outcome
is produced by some convex extrapolation step
within the outer loop. The outer loop steps
will be denoted by $i \in [I_{\mathrm{min}} , I_{\mathrm{max}}]$
and is a function of the inner loop which produces the
dual solution. The iteration steps for the inner loop
will also be denoted by $j \in [J_{\mathrm{min}} , J_{\mathrm{max}}].$ 
The inner loop is not only necessary for 
calculating the dual variable but also for maintaining
the dynamism of Bregman iteration.

For the sake of reading our algorithm  
and the following formulations of the upcoming 
results conveniently,  
we introduce some new notations below
\begin{eqnarray}
\nonumber
\left\{ \begin{array}{rcl}
\hat{u}_{i} & = & \hat{u}_{\alpha_{i}}^{\delta} \\
\hat{u}_{i+1} & = & \hat{u}_{\alpha_{i + 1}}^{\delta} \\
\hat{u}_{i}^{j} & = & \hat{u}_{\alpha_{i}}^{\delta , j} \\
\hat{w}_{i}^{j} & = & \hat{w}_{\alpha_{i}}^{\delta , j} \\
D^{T}(\hat{w}_{i}^{j} - \hat{w}_{i-1}^{j}) & = & 
D^{T}(\hat{w}_{\alpha_{i}}^{\delta , j} - 
\hat{w}_{\alpha_{i-1}}^{\delta , j}).
\end{array}\right.
\end{eqnarray}

\begin{algorithm}
  \caption{ Nested Primal Dual Algorithm with Convex Extrapolation }
  \label{algorithmNPA-I}
  \begin{algorithmic}[1]
    \Procedure{ The iterative regularization procedure to find the minimum 
    of the objective functional (\ref{obj_functional2}) }
    {}
      \State \textbf{initiation} 
     Define the static parameters $\mu,$ $\lambda$ and 
    $\nu.$ Given $u_{0},$ calculate 
    $\hat{w}_{0}^{0} \in \partial \Vert D u_{0}\Vert_1$ and set
        $\hat{w}_{1}^{0} = \hat{w}_{0}^{0}$
      \While{$\underline{\tau}\delta\leq
      \Vert T u_{i+1} - v^{\delta}\Vert_{\mathcal{Y}}\leq
      \overline{\tau}\delta$ or 
      $\Vert u_{i+1} - u^{\dagger}\Vert_{\mathcal{X}}/
      \Vert u^{\dagger}\Vert_{\mathcal{X}} \leq \epsilon$}
        \State Update the regularization parameter $\alpha_{i}$ 
        \For{ $j = 1, \cdots, J$}
         \State $\hat{w}_{0}^{0} = \hat{w}_{0}^{J}$
         \State 
         \label{algorithm_primalstep}
         $\hat{u}_{i}^{j} = \mathrm{prox}_{\mu h}                                                                                                                                   
        \left[u_{i} - \mu\left( T^{T}(T u_{i} - v^{\delta}) + 
        \alpha_i D^{T}(\hat{w}_{i}^{j} - 
        \hat{w}_{0}^{j})\right) \right]$ 
        \State
        \label{algorithm_dualstep}
         $\hat{w}_{i}^{j+1} = 
        \mathrm{prox}_{\nu_i g^{\ast}} 
        \left( \hat{w}_{i}^{j} + \nu D \hat{u}_{i}^{j} \right)$ 
        \State\textbf{calculate} $D^{T}(\hat{w}_{i}^{j} - \hat{w}_{0}^{j})$
      \EndFor
        
        \State \textbf{update} 
        \label{cvx_extrapolation}
        $u_{i + 1} = u_i + \lambda(\hat{u}_i^{J} - u_i) $ 
        \Comment{Convex extrapolation with some $\lambda \in (1,2)$}
        \State $u_0 = u_{i+1}$ and 
        $w_{0}^{j} \in \partial \Vert D u_0\Vert_{1}$
      \EndWhile     

    \EndProcedure
  \end{algorithmic}
\end{algorithm}

Note that the outer loop is not just to update the parameters,
but also to obtain the targeted variable $u_{i+1}$ and
the dual variable $\hat{w}_{0}^{j}.$ 

\subsection{Convergence of the iterative regularization scheme}
\label{iterative_convergence}

We refer the reader to the well-established studies
\textbf{\cite[Section 6]{Engl96},
\cite[Eq. (7.7) \& (7.8)]{HohageSchormann98}}
regarding the concept of iterative regularization. 

Convergence of our iterative variational regularization strategy
must be verified from two aspects: as an iterative scheme and as 
a variational regularization strategy. Therefore, 
here, we study the convergence both whilst the noise amount
$\delta \rightarrow 0$ and the iteration step 
$i \rightarrow \infty.$ 
In analogous with
the continous setting given earlier, there exists
some $i^{\ast} \in (I_{\mathrm{min}} , I_{\mathrm{max}})$
such that 
\begin{displaymath}
\alpha_{i^{\ast}} = \alpha_{i^{\ast}}(\delta , v^{\delta}) 
\rightarrow 0
\mbox{, and } \frac{\delta^2}{\alpha_{i^{\ast}}} \rightarrow 0
\mbox{, as } \delta\rightarrow 0.
\end{displaymath}
Then for any $i \leq i^{\ast},$ the convergence
of the regularized solution to the true solution of the inverse 
problem in the Hadamard sense,
\begin{displaymath}
\Vert\hat{u}_{i^{\ast}} - u^{\dagger}\Vert \rightarrow 0 
\mbox{, as }\delta \rightarrow 0,
\end{displaymath}
furthermore, the convergence in the iterative sense is based
on observing the behaviour of the algorithm's result
during the iteration procedure
\begin{displaymath}
\Vert\hat{u}_{i} - u^{\dagger}\Vert \rightarrow 0 
\mbox{, as } i \rightarrow \infty.
\end{displaymath}
However, from technical point of view, it might be difficult to observe
these rates of convergences separately. Therefore, we will derive
some unified convergence schemce where both concepts
are observed.
Despite its worth, investigation of the number of the iteration steps 
is not in the interest of this work. We will postulate 
the conditions for the choices of the parameters
explicitly so that the Hadamard convergence can be observed.

Note that here we are talking about estimating some 
upper bound for the total error 
estimation per iteration step $i$ that depends on the noise 
level $\delta.$ Iterative total error
estimation can be decomposed in the following form,
\begin{displaymath}
\Vert\hat{u}_{i+1} - u^{\dagger}\Vert \leq 
\Vert\hat{u}_{i+1} - u_{\alpha}^{\delta}\Vert + 
\Vert u_{\alpha}^{\delta} - u^{\dagger}\Vert.
\end{displaymath}
The continuous analysis on the first part of this work
has been dedicated to bound the second term 
on the right hand side. In this  section,
we will establish the necessary and sufficient conditions
to find the bound for the term 
$\Vert\hat{u}_{i+1} - u_{\alpha}^{\delta}\Vert.$
In other words, following up the literature 
\textbf{\cite{ChenLoris18}},
we will prove that with the correct choices
of the parameters, the
iterative solution $\hat{u}_{i+1}$
for the fixed point equation
(\ref{prox_alg}) is the approximation of the regularized
minimizer of the objective functional $F_{\alpha}$
in (\ref{obj_functional2}). Unlike in the aforementioned 
literature, imposing the coercivity of the gradient 
operator $D$ of the TV functional {\em cf. 
\textbf{\cite[Theorem 2]{ChenLoris18}}},
is not necessary.

The primary tool to study stability of the
iteratively regularized solution that is produced by the
proximal gradient algorithm is formulated below.
\begin{ppty}\cite[Lemma 1]{LorisVerhoeven11}
\label{prox_update}
If $x^{+} = \mathrm{prox}_{g}(x^{-} + \Delta)$, then
for any $y \in \mathbb{R}^{N},$
\begin{eqnarray}
\label{prox_update_assertion}
\Vert x^{+} - y\Vert^{2} \leq 
\Vert x^{-} - y\Vert^{2} - 
\Vert x^{+} - x^{-}\Vert^2
+ 2 \langle x^{+} - y , \Delta \rangle + 2 g(y) - 2 g(x^{+}).
\end{eqnarray}
\end{ppty}

\begin{theorem}
\label{IterMinVSMin}
Let the iterative regularization parameter be bounded
by the continuous regularization parameter, 
$\alpha_i < \alpha $ and furthermore 
$\alpha(\delta , v^{\delta}) \rightarrow 0$ as 
$delta \rightarrow 0$ and let $\nu_i = {\alpha}^{1/2}.$ 
Also, let the relaxation parameter satisfy
$\frac{1}{\lambda} < 1 - (\alpha_i - \alpha)$ for
$\lambda \in (1,2).$ Then,
given the condition on the step-length 
$\mu \leq \frac{2}{\Vert T\Vert^2}$ and the initial
guess for the inner loop $\hat{w}_{0} = \hat{w}_{0}^{J},$
the iterative update $u_{i+1}$ converges to
the regularized minimizer $u_{\alpha}^{\delta}$
at some linear rate. That is,
\begin{eqnarray}
\nonumber
\Vert u_{i+1} - u_{\alpha}^{\delta}\Vert \leq  0 
\mbox{, as }i \rightarrow \infty \mbox{, and as }
\alpha \rightarrow 0.
\end{eqnarray}
\end{theorem}

\begin{proof}
First, due to the last update step in 
(\ref{cvx_extrapolation}), we apply the equality in 
(\ref{strct_cvx_eq}) as follows,
\begin{eqnarray}
\label{strct_cvx_eq2}
\Vert u_{i+1} - u_{\alpha}^{\delta}\Vert^{2} = 
(1 - \lambda)\Vert u_{i} - u_{\alpha}^{\delta}\Vert^2 + 
\lambda\Vert\hat{u}_{i}^{J} - u_{\alpha}^{\delta}\Vert 
-\lambda(1 - \lambda)
\Vert u_i - \hat{u}_{i}^{J}\Vert^2.
\end{eqnarray}
Boundedness of each term here must be guaranteed.
Assuring that the primal term $\hat{u}_{i}^{J}$ is the 
approximation of the regularized minimum $u_{\alpha}^{\delta}$
will convey the reasons behind the choices of the parameters.
To this end, applying the Property \ref{prox_update} on the
primal variable $\hat{u}_{i}^{J}$ produced at the final outer 
loop step $J$ returns
\begin{eqnarray}
\Vert\hat{u}_{i}^{J} - u_{\alpha}^{\delta}\Vert^2 \leq 
\Vert u_{i} - u_{\alpha}^{\delta}\Vert - 
\Vert \hat{u}_{i}^{J} - u_{i}\Vert^2
& - & 2\mu\langle \hat{u}_{i}^{J} - u_{\alpha}^{\delta} ,  
T^{T}(T u_{i} - v^{\delta}) \rangle 
\nonumber\\
& - & 2\mu\alpha_{i}
\langle \hat{u}_{i}^{J} - u_{\alpha}^{\delta} ,  
D^{T}(\hat{w}_{i}^{J} - \hat{w}_{0}^{J}) \rangle .
\nonumber
\end{eqnarray}
Likewise, application of the Property \ref{prox_update} on the
regularized minimizer $u_{\alpha}^{\delta}$ yields
\begin{eqnarray}
\Vert u_{\alpha}^{\delta} - \hat{u}_{i}^{J}\Vert^2 \leq 
\Vert u_{\alpha}^{\delta} - \hat{u}_{i}^{J}\Vert^2 - 
\Vert u_{\alpha}^{\delta} - u_{\alpha}^{\delta}\Vert^2 
& - & 2\mu\langle u_{\alpha}^{\delta} - \hat{u}_{i}^{J} , 
T^{T}(T u_{\alpha}^{\delta} - v^{\delta}) \rangle
\nonumber\\
& - & 2\mu\alpha\langle u_{\alpha}^{\delta} - \hat{u}_{i}^{J} , 
D^{T}(\hat{w}_{\alpha}^{\delta} - w_0)\rangle
\nonumber 
\end{eqnarray}
If we sum up the last two estimations, then 
\begin{eqnarray}
\Vert\hat{u}_{i}^{J} - u_{\alpha}^{\delta}\Vert^2 \leq 
\Vert u_i - u_{\alpha}^{\delta}\Vert^2 - 
\Vert\hat{u}_{i}^{J} - u_{i}\Vert^2 
& + & 2\mu\langle u_{\alpha}^{\delta} - \hat{u}_{i}^{J} , 
T^{T}T(u_i - u_{\alpha}^{\delta}) \rangle
\nonumber\\
& - &2\mu\alpha_{i}\langle \hat{u}_{i}^{J} - 
u_{\alpha}^{\delta} , 
D^{T}(\hat{w}_{i}^{J} - \hat{w}_{0}^{J})\rangle 
\nonumber\\
& - & 2\mu\alpha\langle u_{\alpha}^{\delta} - \hat{u}_{i}^{J} , 
D^{T}(\hat{w}_{\alpha}^{\delta} - w_0)\rangle .
\nonumber 
\label{primVSregmin1}
\end{eqnarray}
Let us set $I = \langle u_{\alpha}^{\delta} - \hat{u}_{i}^{j} , 
T^{T}T(u_i - u_{\alpha}^{\delta}) \rangle$ and
rewrite the first inner product on the right hand side,
\begin{eqnarray}
I = \langle u_{\alpha}^{\delta} - u_{i} , 
T^{T}T(u_i - u_{\alpha}^{\delta}) \rangle + 
\langle u_{i} - \hat{u}_{i}^{J} , 
T^{T}T(u_i - u_{\alpha}^{\delta}) \rangle .
\nonumber
\end{eqnarray}
Let us do further assignments for the inner products,
\begin{eqnarray}
I_{1} & = & - \langle u_{i} - u_{\alpha}^{\delta} , 
T^{T}T(u_i - u_{\alpha}^{\delta}) \rangle
\nonumber\\
I_{2} & = & \langle u_{i} - \hat{u}_{i}^{J} , 
T^{T}T(u_i - u_{\alpha}^{\delta}) \rangle .
\nonumber
\end{eqnarray} 
The early estimation in (\ref{LipschitzEstMisfit})
is applied to bound $I_1$ as such,
\begin{eqnarray}
I_{1} \leq - \frac{1}{\Vert T\Vert^2}
\Vert T^{T}T(u_{i} - u_{\alpha}^{\delta})\Vert^2.
\nonumber
\end{eqnarray}
Also, the term $I_2$ can be bounded by using the simple identity, 
$0 \leq \Vert a \vec{x} - b \vec{y}\Vert^2 = 
a^2\Vert\vec{x}\Vert^2 - 
2ab \langle \vec{x} , \vec{y} \rangle + 
b^2\Vert\vec{y}\Vert^2,$
for $a,b \in \mathbb{R}_{+},$
\begin{eqnarray}
I_2 \leq \frac{1}{2}\Vert u_{i} - \hat{u}_{i}^{J}\Vert^2
+ \frac{1}{2}\Vert T^{T}T(u_{i} - u_{\alpha}^{\delta})\Vert^2.
\nonumber
\end{eqnarray}
Then, with both of the bounds on $I_1$ and $I_2,$
(\ref{primVSregmin1}) reads
\begin{eqnarray}
\Vert\hat{u}_{i}^{J} - u_{\alpha}^{\delta}\Vert^2 \leq 
\Vert u_i - u_{\alpha}^{\delta}\Vert^2 - 
\Vert\hat{u}_{i}^{J} - u_{i}\Vert^2 
& - &\mu\frac{1}{\Vert T\Vert^2}
\Vert T^{T}T(u_{i} - u_{\alpha}^{\delta})\Vert^2
\nonumber\\
& + & \frac{1}{2}\Vert u_{i} - \hat{u}_{i}^{J}\Vert^2 +
\mu^{2} \frac{1}{2}\Vert T^{T}T(u_{i} - u_{\alpha}^{\delta})\Vert^2
\nonumber\\
& - & 2\mu\alpha_{i}\langle \hat{u}_{i}^{J} - 
u_{\alpha}^{\delta} , 
D^{T}(\hat{w}_{i}^{J} - \hat{w}_{0}^{J})\rangle ,
\nonumber
\end{eqnarray}
which is in other words,
\begin{eqnarray}
\Vert\hat{u}_{i}^{J} - u_{\alpha}^{\delta}\Vert^2 \leq 
\Vert u_i - u_{\alpha}^{\delta}\Vert^2 - 
\frac{1}{2}\Vert u_{i} - \hat{u}_{i}^{J}\Vert^2 
& + &\mu\left(\mu\frac{1}{2} - \frac{1}{\Vert T\Vert^2}\right)
\Vert T^{T} T(u_{i} - u_{\alpha}^{\delta})\Vert^2 
\nonumber\\
& - &2\mu\alpha_{i}\langle \hat{u}_{i}^{J} - 
u_{\alpha}^{\delta} , 
D^{T}(\hat{w}_{i}^{J} - \hat{w}_{0}^{J})\rangle
\nonumber\\
& - & 2\mu\alpha\langle u_{\alpha}^{\delta} - \hat{u}_{i}^{J} , 
D^{T}(\hat{w}_{\alpha}^{\delta} - w_0)\rangle .
\label{primVSregmin2}
\end{eqnarray}
Now, we apply Property \ref{prox_update}
for the dual variables $\hat{w}_{i}^{J}$ and 
$\hat{w}_{\alpha}^{\delta}.$ First, by the definition of
$\hat{w}_{\alpha}^{\delta}$ in (\ref{eq_subgrad_char}),
\begin{eqnarray}
\Vert\hat{w}_{\alpha}^{\delta} - \hat{w}_{i}^{J}\Vert^2 \leq 
\Vert\hat{w}_{\alpha}^{\delta} - \hat{w}_{i}^{J}\Vert^2 -
\Vert\hat{w}_{\alpha}^{\delta} - \hat{w}_{\alpha}^{\delta}\Vert^2 
+ 2\nu\langle \hat{w}_{\alpha}^{\delta} - \hat{w}_{i}^{J} , 
D u_{\alpha}^{\delta} \rangle + 2\nu g^{\ast}(\hat{w}_{i}^{J})
- 2\nu g^{\ast}(\hat{w}_{\alpha}^{\delta}),
\nonumber
\end{eqnarray}
which is,
\begin{eqnarray}
0 \leq 2\nu\langle \hat{w}_{\alpha}^{\delta} - \hat{w}_{i}^{J} , 
D u_{\alpha}^{\delta} \rangle + 2\nu g^{\ast}(\hat{w}_{i}^{J})
- 2\nu g^{\ast}(\hat{w}_{\alpha}^{\delta}).
\nonumber
\end{eqnarray}
If we multiply both sides by $\frac{\nu_{i}}{\nu} > 0,$ then
\begin{eqnarray}
0 \leq 2\nu_{i}\langle \hat{w}_{\alpha}^{\delta} - \hat{w}_{i}^{J} , 
D u_{\alpha}^{\delta} \rangle + 2\nu_{i} g^{\ast}(\hat{w}_{i}^{J})
- 2\nu_{i} g^{\ast}(\hat{w}_{\alpha}^{\delta}),
\label{dualVSdual1}
\end{eqnarray}
which will be more beneficial to our analysis soon.
Also, the definition of $\hat{w}_{i}^{J}$ implies,
\begin{eqnarray}
\Vert\hat{w}_{i}^{J} - \hat{w}_{\alpha}^{\delta}\Vert^2 \leq 
\Vert\hat{w}_{i}^{J-1} - \hat{w}_{\alpha}^{\delta}\Vert^2 - 
\Vert\hat{w}_{i}^{J} - \hat{w}_{i}^{J-1}\Vert^2 + 
2\nu_{i}\langle \hat{w}_{i}^{J} - \hat{w}_{\alpha}^{\delta} , 
D\hat{u}_{i}^{J} \rangle + 2\nu_{i} g^{\ast}(\hat{w}_{\alpha}^{\delta})
- 2\nu_{i} g^{\ast}(\hat{w}_{i}^{J})
\nonumber\\
\label{dualVSdual2}
\end{eqnarray}
Both (\ref{dualVSdual1}) and (\ref{dualVSdual2}) in total
will bring
\begin{eqnarray}
\Vert\hat{w}_{i}^{J} - \hat{w}_{\alpha}^{\delta}\Vert^2 \leq 
\Vert\hat{w}_{i}^{J-1} - \hat{w}_{\alpha}^{\delta}\Vert^2 - 
\Vert\hat{w}_{i}^{J} - \hat{w}_{i}^{J-1}\Vert^2 + 
2\nu_{i}\langle \hat{w}_{\alpha}^{\delta} - \hat{w}_{i}^{J} , 
D(u_{\alpha}^{\delta} - \hat{u}_{i}^{J}) \rangle . 
\label{dualVSdual3}
\end{eqnarray}
Now, $\nu_{i}$ times (\ref{primVSregmin2}) and $\mu\alpha$
times (\ref{dualVSdual3}),
\begin{eqnarray}
\nu_{i}\Vert\hat{u}_{i}^{J} - u_{\alpha}^{\delta}\Vert^2 + 
\mu\alpha\Vert\hat{w}_{i}^{J} - \hat{w}_{\alpha}^{\delta}\Vert^2 
& \leq &\nu_{i}\Vert u_{i} - u_{\alpha}^{\delta}\Vert^2
+\mu\alpha\Vert\hat{w}_{i}^{J - 1} - \hat{w}_{\alpha}^{\delta}\Vert^2
\nonumber\\
& - &\nu_{i}\frac{1}{2}\Vert u_{i} - \hat{u}_{i}^{J}\Vert^2 
-\mu\alpha \Vert\hat{w}_{i}^{J} - \hat{w}_{i}^{J-1}\Vert^2
\nonumber\\
& + &\mu\nu_{i}\left(\mu\frac{1}{2} - \frac{1}{\Vert T\Vert^2}\right)
\Vert T^{T}T(u_{i} - u_{\alpha}^{\delta})\Vert^2
\nonumber\\
& + &2\mu\alpha\nu_{i}\langle \hat{w}_{\alpha}^{\delta} - \hat{w}_{i}^{J} , 
D(u_{\alpha}^{\delta} - \hat{u}_{i}^{J}) \rangle
\nonumber\\
& - & 2\mu\alpha_{i}\nu_{i}\langle \hat{u}_{i}^{J} - u_{\alpha}^{\delta} , 
D^{T}(\hat{w}_{i}^{j} - \hat{w}_{0}^{J}) \rangle
\nonumber\\
& - &2\mu\alpha\nu_{i}\langle u_{\alpha}^{\delta} - \hat{u}_{i}^{J} , 
D^{T}(\hat{w}_{\alpha}^{\delta} - w_0)\rangle .
\label{primalWdual1}
\end{eqnarray}
Note that the bound for the term 
$\Vert\hat{u}_{i}^{J} - u_{\alpha}^{\delta}\Vert^2$ is in the best
interest of this analysis. So, we will drop the term
$\mu\alpha\Vert\hat{w}_{i}^{J} - \hat{w}_{\alpha}^{\delta}\Vert^2$
on the left hand side. Also, after quick calculations, with 
$\langle u , D^{T}w \rangle = \langle D u , w \rangle,$ 
on the right hand side, 
the inner products will be simplified. Thus,
(\ref{primalWdual1}) reads
\begin{eqnarray}
\nu_{i}\Vert\hat{u}_{i}^{J} - u_{\alpha}^{\delta}\Vert^2 & \leq &
\nu_i\Vert u_{i} - u_{\alpha}^{\delta}\Vert^2
+\mu\alpha\Vert\hat{w}_{i}^{J - 1} - \hat{w}_{\alpha}^{\delta}\Vert^2
-\nu_{i}\frac{1}{2}\Vert u_{i} - \hat{u}_{i}^{J}\Vert^2 
-\mu\alpha \Vert\hat{w}_{i}^{J} - \hat{w}_{i}^{J-1}\Vert^2
\nonumber\\
& + &\mu\nu_{i}\left(\mu\frac{1}{2} - \frac{1}{\Vert T\Vert^2}\right)
\Vert T^{T}T(u_{i} - u_{\alpha}^{\delta})\Vert^2
\nonumber\\
& + &2\mu(\alpha_{i} - \alpha)\nu_{i}
\langle D(u_{\alpha}^{\delta} - \hat{u}_{i}^{J}) , 
\hat{w}_{\alpha}^{\delta} - w_0 \rangle
+ 2\mu\alpha\nu_{i}
\langle D(u_{\alpha}^{\delta} - \hat{u}_{i}^{J}) , 
w_0 - \hat{w}_{0}^{J}\rangle .
\nonumber
\end{eqnarray}
Here, since the term $\hat{w}_{0}^{J}$ is assigned to
$\hat{w}_{0}$, the second inner product on the right hand
side drops. Furthermore, by the choice of the step-length,
the 5th term on the right remains negative. 
Also, we can ignore other negative terms on the right hand side.
Multplication of both sides by $\frac{1}{\nu_i}$ gives
\begin{eqnarray}
\Vert\hat{u}_{i}^{J} - u_{\alpha}^{\delta}\Vert^2 & \leq &
\Vert u_{i} - u_{\alpha}^{\delta}\Vert^2
+\frac{1}{\nu_i}\mu\alpha
\Vert\hat{w}_{i}^{J - 1} - \hat{w}_{\alpha}^{\delta}\Vert^2
\nonumber\\
& + & 2\mu(\alpha_{i} - \alpha)
\langle D(u_{\alpha}^{\delta} - \hat{u}_{i}^{J}) , 
\hat{w}_{\alpha}^{\delta} - w_0 \rangle .
\nonumber
\end{eqnarray}
Thus far, we have analysed the boundedness of the term 
$\Vert\hat{u}_{i}^{J} - u_{\alpha}^{\delta}\Vert$ in 
(\ref{strct_cvx_eq2}). So, we have obtained,
\begin{eqnarray}
\Vert u_{i+1} - u_{\alpha}^{\delta}\Vert^2 \leq 
\Vert u_{i} - u_{\alpha}^{\delta}\Vert^2 
& + &\lambda\frac{1}{\nu_i}\mu\alpha
\Vert\hat{w}_{i}^{J-1} - \hat{w}_{\alpha}^{\delta}\Vert 
 - \lambda(1 - \lambda)\Vert u_i - \hat{u}_{i}^{J}\Vert^2
\nonumber\\
& + & 2\lambda\mu(\alpha_{i} - \alpha)
\langle D(u_{\alpha}^{\delta} - \hat{u}_{i}^{J}) , 
\hat{w}_{\alpha}^{\delta} - w_0 \rangle .
\label{primVSregmin3}
\end{eqnarray}
Appropriate decomposition for the inner product on the right
hand side can be given as
\begin{eqnarray}
\langle D(u_{\alpha}^{\delta} - \hat{u}_{i}^{J}) , 
\hat{w}_{\alpha}^{\delta} - w_0 \rangle = 
\langle D(u_{\alpha}^{\delta} - u_{i}) , 
\hat{w}_{\alpha}^{\delta} - w_0 \rangle + 
\langle D(u_{i} - \hat{u}_{i}^{J}) , 
\hat{w}_{\alpha}^{\delta} - w_0 \rangle .
\nonumber
\end{eqnarray}
Then, one can bound the each individual inner products,
\begin{eqnarray}
2\lambda\mu(\alpha_{i} - \alpha)
\langle u_{\alpha}^{\delta} - u_{i} , 
D^{T}(\hat{w}_{\alpha}^{\delta} - w_0) \rangle & \leq & 
\lambda^2(\alpha_i - \alpha)\Vert u_{\alpha}^{\delta} - u_{i}\Vert^2
+ \mu^2(\alpha_i - \alpha)
\Vert D^{T}(\hat{w}_{\alpha}^{\delta} - w_0)\Vert^2,
\nonumber\\
2\lambda\mu(\alpha_{i} - \alpha)
\langle u_{i} - \hat{u}_{i}^{J} , 
D^{T}(\hat{w}_{\alpha}^{\delta} - w_0) \rangle & \leq & 
\lambda^2(\alpha_i - \alpha)\Vert u_{i} - \hat{u}_{i}^{J}\Vert^2
+ \mu^2(\alpha_i - \alpha)
\Vert D^{T}(\hat{w}_{\alpha}^{\delta} - w_0)\Vert^2.
\nonumber
\end{eqnarray}
Eventually, (\ref{primVSregmin3}) takes the following form
\begin{eqnarray}
\Vert u_{i+1} - u_{\alpha}^{\delta}\Vert^2 & \leq & 
(1 + \lambda^2(\alpha_i - \alpha))
\Vert u_i - u_{\alpha}^{\delta}\Vert^2 
+\lambda\left(\lambda(\alpha_i - \alpha) - (1 - \lambda)\right)
\Vert u_i - \hat{u}_{i}^{J}\Vert^2 
\nonumber\\
& + &\lambda\frac{1}{\nu_i}\mu\alpha
\Vert\hat{w}_{i}^{J-1} - \hat{w}_{\alpha}^{\delta}\Vert + 
2\mu^2(\alpha_i - \alpha)
\Vert D^{T}(\hat{w}_{\alpha}^{\delta} - w_0)\Vert^2.
\label{primVSregmin4}
\end{eqnarray}
Conditions on the parameters ensure the linear convergence
rate as follows; Since $\alpha_i < \alpha$ the factor
$(1 + \lambda^2(\alpha_i - \alpha)) \in (0,1).$ On the other hand,
the relaxation parameter stisfies
$\frac{1}{\lambda} < 1 - (\alpha_i - \alpha)$ for
$\lambda \in (1,2).$ Then, the factor
$\left(\lambda(\alpha_i - \alpha) - (1 - \lambda)\right)$ remains
negative. So that the terms on the right hand side
$\Vert u_i - \hat{u}_{i}^{J}\Vert^2$ and 
$\Vert D^{T}(\hat{w}_{\alpha}^{\delta} - w_0)\Vert^2$ can drop. 
Lastly, the choice of $\nu_{i}$ ensures the convergence.

\end{proof}

\begin{remark}
In the statement of the proof, although the
the parameter $\nu_{i}$ has been introduced as a function
of the the iteration step $i,$ it is set
to the upper bound of the iterative regularization 
parameter. This choice of the parameter is rather convenient
for the sake of the analysis. In the numerical tests,
this means that the parameter $\nu_{i}$ remains fixed
during the whole procedure.
\end{remark}

Now, we have come to the point where we can show that the
iteratively regularized approximation converges to the
minimum norm solution $u^{\dagger}.$ To this end, in addition
to the parameters defined in Theorem \ref{IterMinVSMin},
the VSC in Assumption 
\ref{assump_conventional_variational_ineq} 
is also needed.

\begin{theorem}
\label{IterMinVS_Exact}
Let the minimum norm solution 
$u^{\dagger} \in \mathrm{BV}(\Omega)$
satisfy the Assumption 
\ref{assump_conventional_variational_ineq}
for the linear operator equation
$T u^{\dagger} = v^{\dagger}$
where $T : \mathbb{R}^{N} \rightarrow \mathbb{R}^{M}$ 
and $v^{\delta} \in \mathcal{B}_{\delta}(v^{\dagger}).$
Given the {\em a-posteriori} choice of the dynamical 
regularization parameter 
$\alpha_{i}  = \frac{1}{i(\delta , v^{\delta})},$ 
the relaxation parameter $\lambda \in (1,2)$ 
and $u_{i} \in BV(\Omega),$
the iteratively regularized approximate minimizer 
of the objective functional (\ref{obj_functional2})
$u_{i+1}$ converges
to the minimum norm solution $u^{\dagger}$
in the Hadamard sense,
\begin{displaymath}
\Vert u_{i+1} - u^{\dagger}\Vert \rightarrow 0
\mbox{, as } 
i(\delta , v^{\delta}) \rightarrow \infty 
\mbox{ whilst }\delta \rightarrow 0 .
\end{displaymath}
\end{theorem}

\begin{proof}
As has been done above, we begin with
applying the equality in 
(\ref{strct_cvx_eq}),
\begin{eqnarray}
\label{strct_cvx_eq3}
\Vert u_{i+1} - u^{\dagger}\Vert^{2} = 
(1 - \lambda)\Vert u_{i} - u^{\dagger}\Vert^2 + 
\lambda\Vert\hat{u}_{i}^{J} - u^{\dagger}\Vert 
-\lambda(1 - \lambda)\Vert u_i - \hat{u}_{i}^{J}\Vert^2.
\end{eqnarray}
Analogously, we investigate the boundedness of the term
$\Vert\hat{u}_{i}^{J} - u^{\dagger}\Vert$ by using the 
Property~\ref{prox_update},
\begin{eqnarray}
\Vert\hat{u}_{i}^{J} - u^{\dagger}\Vert^2 \leq 
\Vert u_{i} - u^{\dagger}\Vert - 
\Vert\hat{u}_{i}^{J} - u_{i}\Vert^2 
& - &2\mu\langle \hat{u}_{i}^{J} - u^{\dagger} , 
T^{T}(T u_{i} - v^{\delta}) \rangle
\nonumber\\
& - &2\mu\alpha_{i}\langle \hat{u}_{i}^{J} - u^{\dagger} , 
D^{T}(\hat{w}_{i}^{J} - \hat{w}_{0}^{J}) \rangle .
\label{est_IterMinVS_Exact}
\end{eqnarray}
By means of the deterministic noise model 
$\Vert v^{\dagger} - v^{\delta}\Vert\leq\delta,$ 
the condition on the step-length
$\mu \leq \frac{2}{ \Vert T \Vert^2},$ and lastly
Assumption \ref{assump_conventional_variational_ineq} 
for the term $\hat{u}_{i}^{J}$ since by definition 
$\hat{u}_{i}^{J} \in BV(\Omega),$
the first inner product on the right hand side 
can be bounded as follows,
\begin{eqnarray}
- 2\mu \langle \hat{u}_{i}^{J} - u^{\dagger} , 
T^{T}(T u_{i} - v^{\delta})\rangle & = & -2\mu
\langle \hat{u}_{i}^{J} - u^{\dagger} , 
T^{T}(T u_{i} - T u^{\dagger}) \rangle
-2\mu\langle \hat{u}_{i}^{J} - u^{\dagger} , 
T^{T}(T u^{\dagger} -v^{\delta}) \rangle 
\nonumber\\
& = &-2\mu \langle \hat{u}_{i}^{J} - u_{i} , 
T^{T}(T u_{i} - Tu^{\dagger}) \rangle
-2\mu\langle u_{i} - u^{\dagger} , 
T^{T}(Tu_{i} - Tu^{\dagger}) \rangle 
\nonumber\\
& - & 2\mu\langle \hat{u}_{i}^{J} - u^{\dagger} , 
T^{T}(T u^{\dagger} -v^{\delta}) \rangle
\nonumber\\
&\leq &2\mu\Vert T\Vert^2\Vert\hat{u}_{i}^{J} - u^{\dagger}\Vert 
\Vert u_{i} - u^{\dagger}\Vert -\mu\Vert Tu_{i} - Tu^{\dagger}\Vert^2
\nonumber\\
& + & 2\delta\mu\Vert T\Vert\Vert\hat{u}_{i}^{J} - u^{\dagger}\Vert
\nonumber\\
& \leq &\Psi(\delta) + \frac{4}{\Vert T \Vert}\delta\Psi(\delta),
\end{eqnarray}
where we have dropped the negative term 
$-\mu\Vert Tu_{i} - Tu^{\dagger}\Vert^2.$

At this point, it is rather redundant to bound the second
inner product on the right hand side of (\ref{est_IterMinVS_Exact})
since similar bound has been found in the proof of 
Theorem~\ref{IterMinVSMin}. Eventually, after inserting
all the expected bounds into the equality (\ref{strct_cvx_eq3}),
the term $\Vert u_{i+1} - u^{\dagger}\Vert$
will be bounded in terms of the index function $\Psi(\delta)$
since $u_{i} \in BV(\Omega).$
\end{proof}

\begin{remark}
Note that unlike in the continuous analysis, for the development
of the iterative analysis, we have not imposed any initial guess
for the iterative procedure. We have completely relied on the
choice of the parameters.
\end{remark}

\section{Numerical Results}
In this section, we will develop numerical results 
with sufficiently low amount of noise. 
The decay on the iterative error estimations
will be displayed both on the image and the pre-image spaces.

The algorithm will be applied for solving a simple 2-D image processing problem. In the first test, the simulated 
measurement data is calculated by some sinogram projections 
of a rank deficient forward operator, 
Figure \ref{2D_star_W_sinogram}. 
In the second test, a full-rank forward operator
is applied on the same test image to calculate the sinogram
measurements, Figure \ref{2D_star_W_sinogram2}. 
The number of the iterations has been defined manually.
Numerical results of each test are profiled in Figure
\ref{2D_star_conv} and Figure \ref{2D_star_conv2}. We also
observe the behaviour of the error estimation with an 
increased amount of noise added on the measurement, Figure 
\ref{2D_star_W_sinogram3} and Figure~\ref{2D_star_conv3}.


\begin{figure}[!tbp]
\includegraphics[height=4in,width=7.5in,angle=0]
{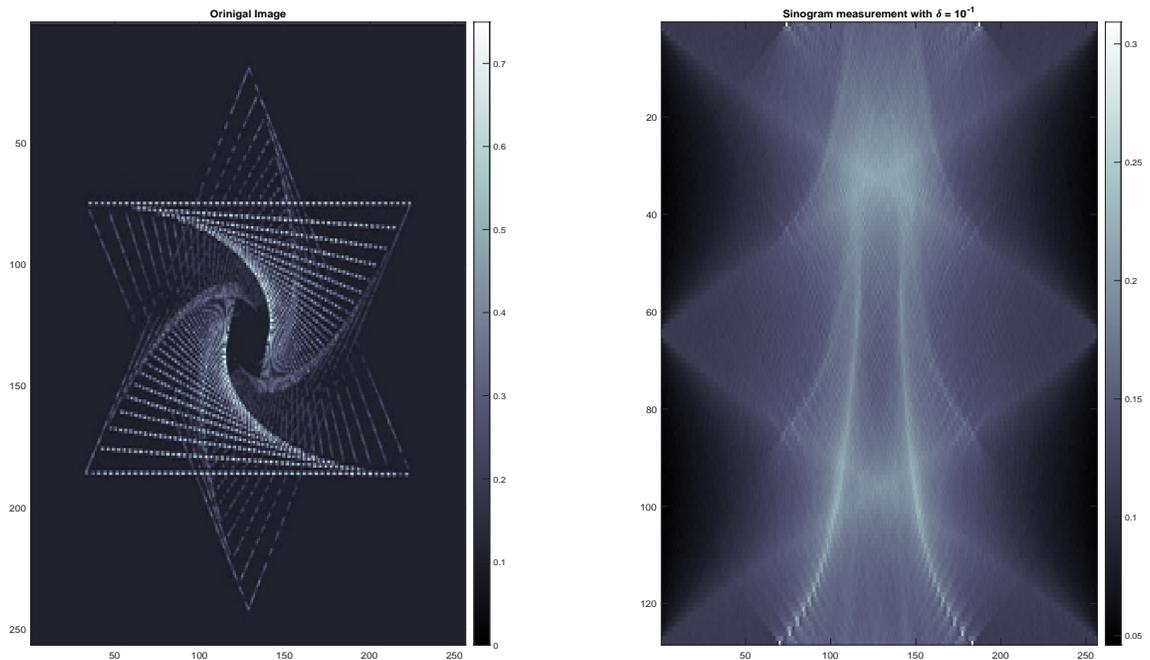}
\caption{Data visualization:}
{\footnotesize Original image and its noisy sinogram 
measurement with $\delta = 0.1 \%$ from a rank deficient 
forward operator with dimensions $32768 \times 65536$.}
\label{2D_star_W_sinogram}
\end{figure}

\begin{figure}[!tbp]
\includegraphics[height=4in,width=7.5in,angle=0]
{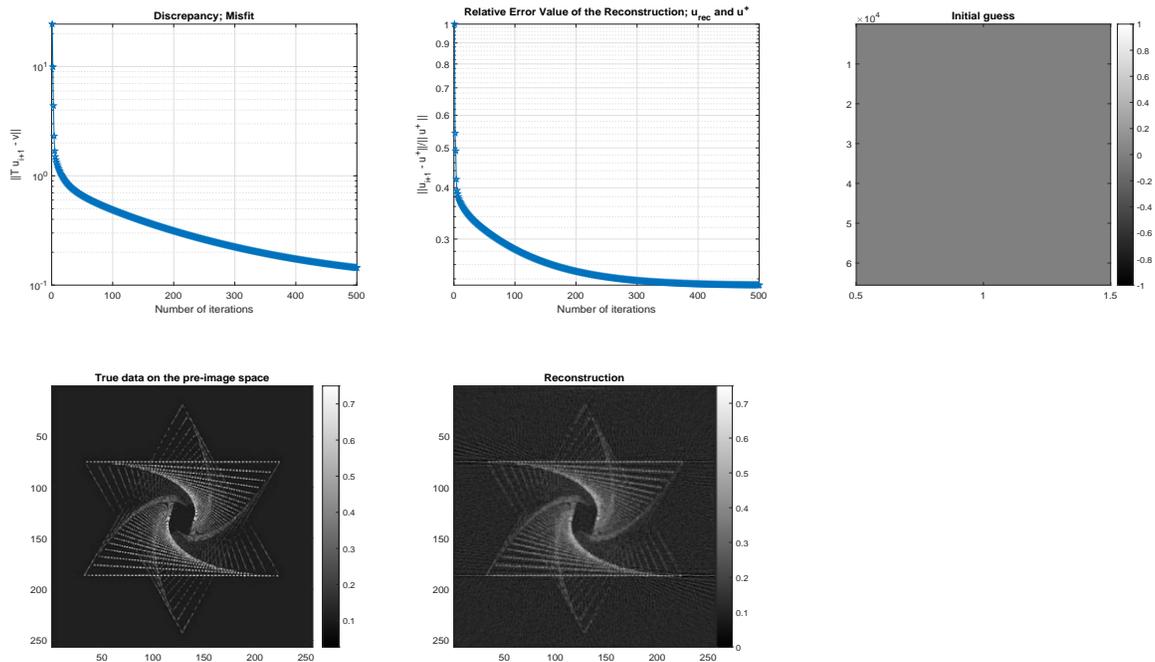}
\caption{Error analysis profiles and data visualization 
with $\delta = 0.1 \%$ and matrix with the dimension 
of $32768 \times 65536$:}
{\footnotesize Numerical results of the algorithm after it is
applied on the underdetermined problem. In the reconstruction, 
we observe some artifacts as the traces of the lines that result
from insufficient number of measurements.
However, edges of the image are preserved due to
TV penalization on the reconstruction. Relative error value of the 
final reconstruction falls below $0.3$.}
\label{2D_star_conv}
\end{figure}


\begin{figure}[!tbp]
\includegraphics[height=4in,width=7.5in,angle=0]
{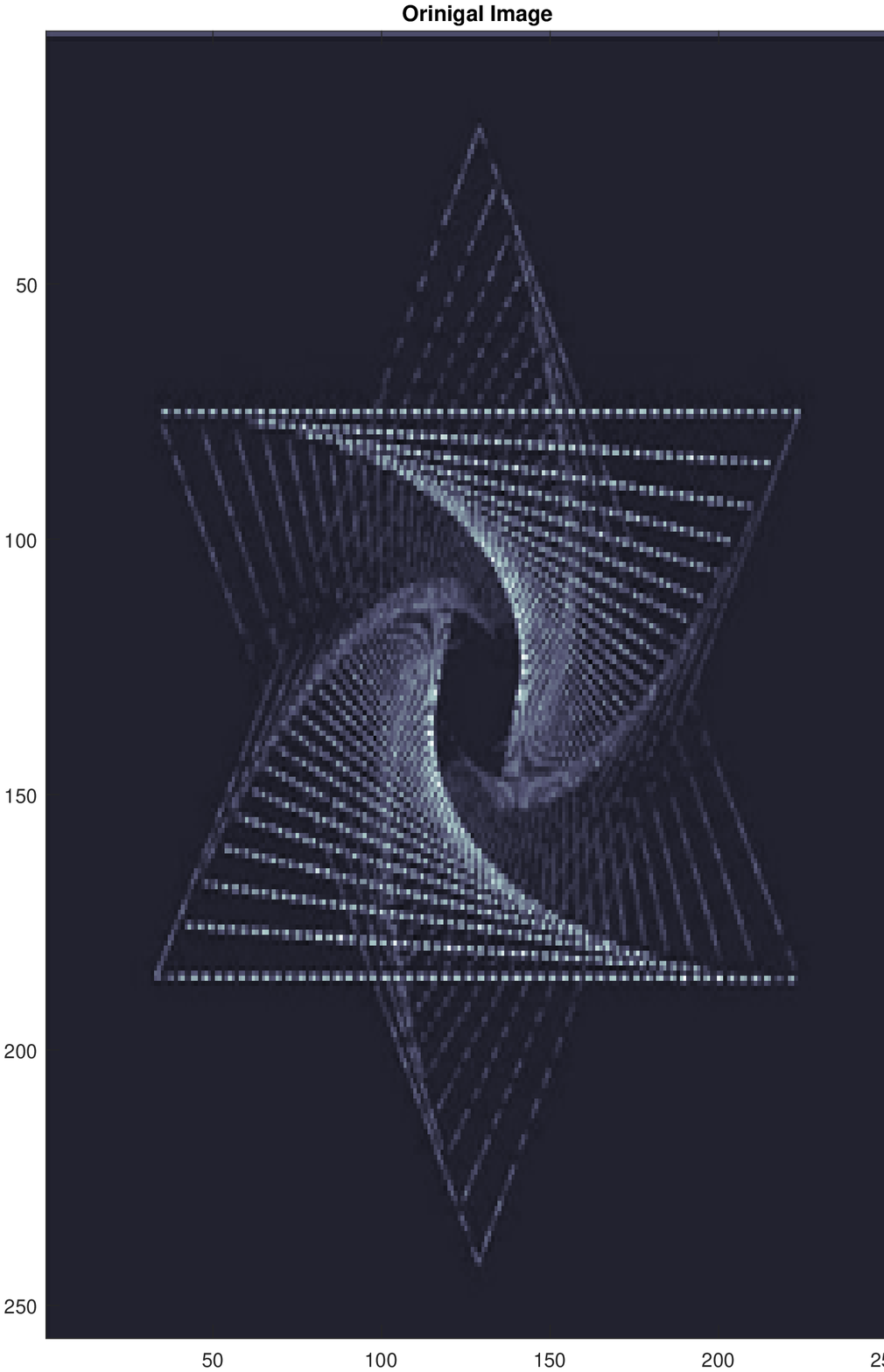}
\caption{Data visualization:}
{\footnotesize Original image and its noisy sinogram 
measurement with $\delta = 0.1 \%$ from a full rank 
forward operator with the dimension 
of $65536 \times 65536$.}
\label{2D_star_W_sinogram2}
\end{figure}

\begin{figure}[!tbp]
\includegraphics[height=4in,width=7.5in,angle=0]
{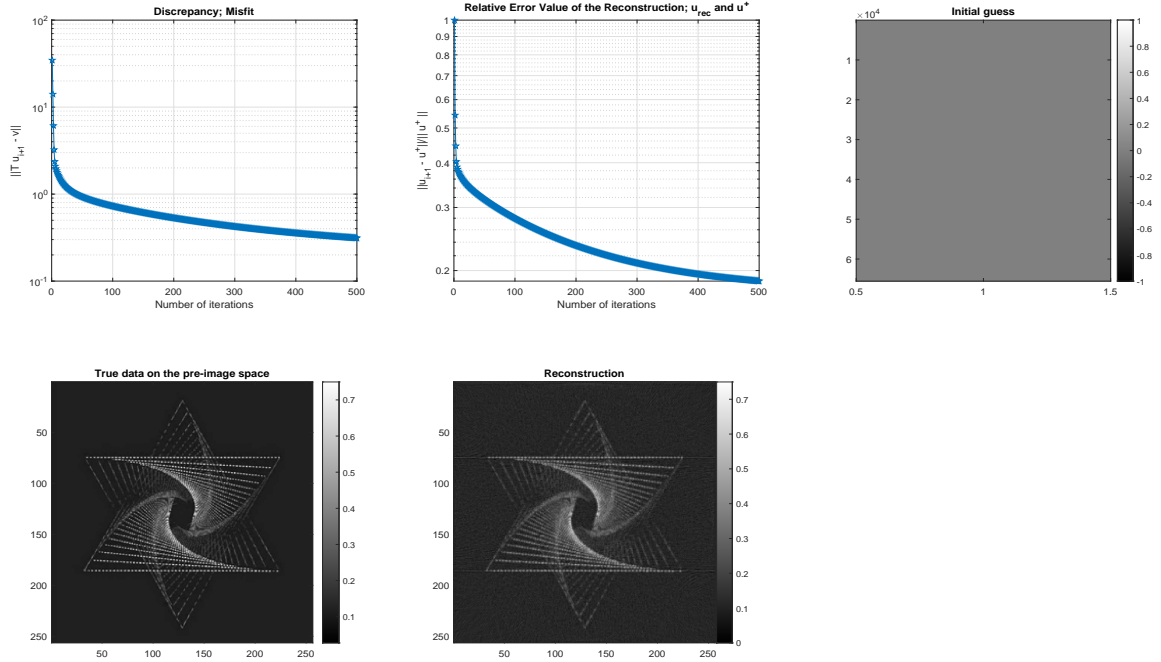}
\caption{Error analysis profiles and data visualization with 
$\delta = 0.1 \%$ from a full-rank forward operator 
with dimensions $65536 \times 65536$:}
{\footnotesize Unlike in the 
underdetermined system in Figure \ref{2D_star_conv}, 
relative error value
of the final reconstruction falls below $0.2$. Also artifacts
are less visible, and again edge preservation by TV penalization 
is observed.}
\label{2D_star_conv2}
\end{figure}


\begin{figure}[!tbp]
\includegraphics[height=4in,width=7.5in,angle=0]
{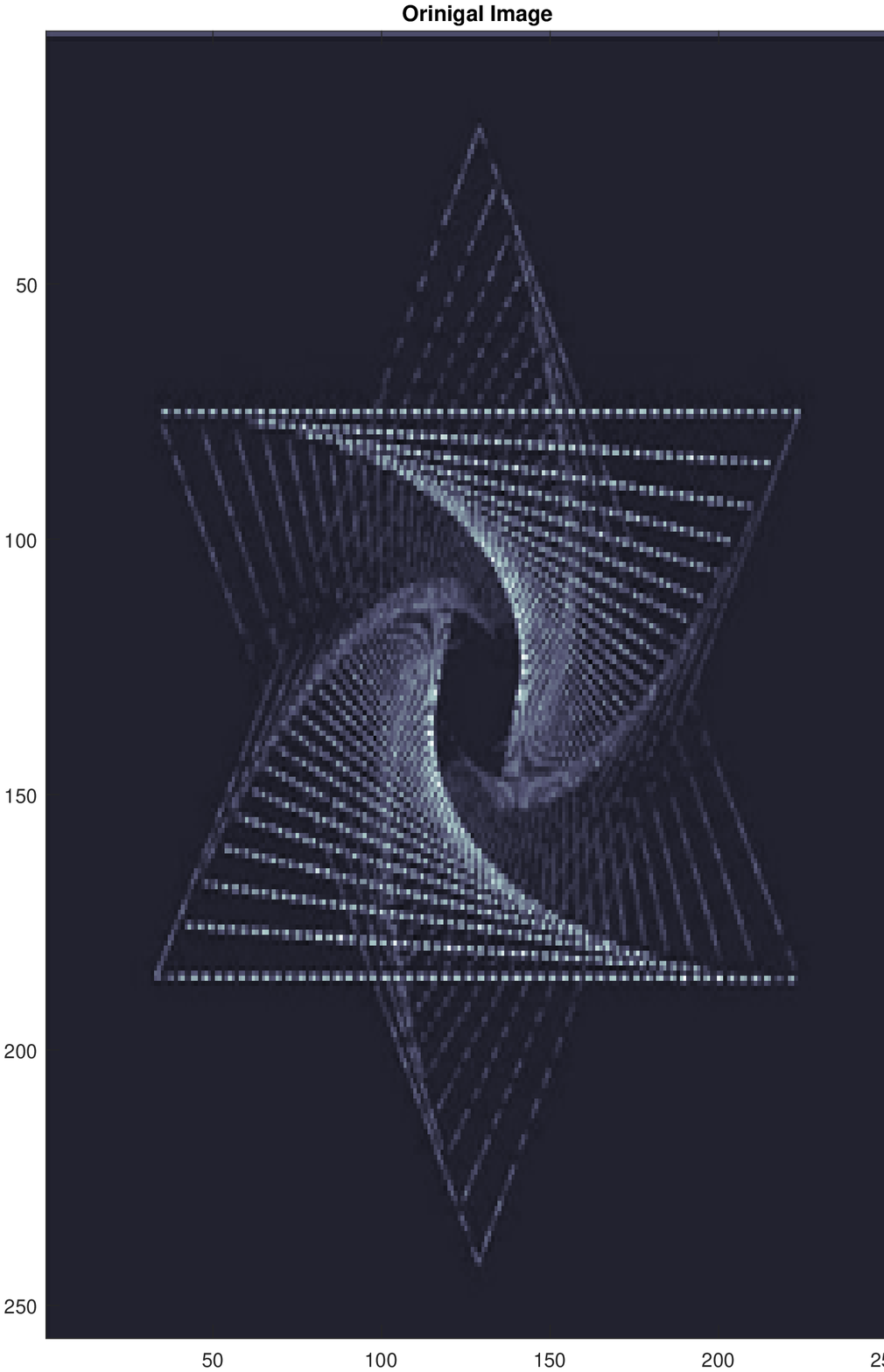}
\caption{Data visualization:}
{\footnotesize Original image and its noisy sinogram 
measurement with $\delta = 1 \% $ from a full-rank 
matrix with dimensions $65536 \times 65536$.}
\label{2D_star_W_sinogram3}
\end{figure}

\begin{figure}[!tbp]
\includegraphics[height=4in,width=7.5in,angle=0]
{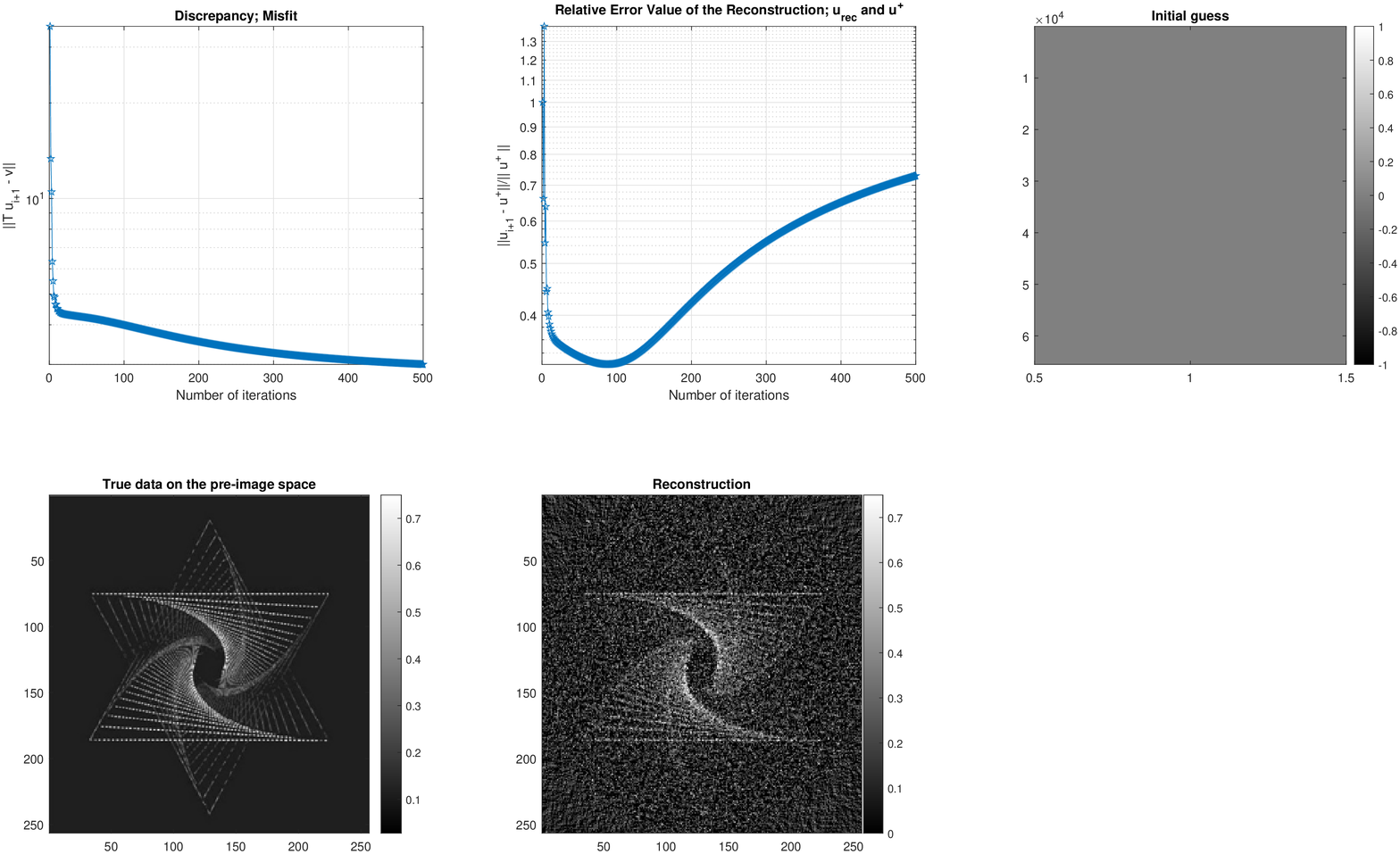}
\caption{Error analysis profiles and data visualization:}
{\footnotesize Numerical results of the algorithm  
for a measurement with
$\delta = 1\% $ from a full-rank 
forward operator with dimensions 
of $65536 \times 65536$. Since the noise level is high,
the relative error estimation 
of the reconstructed
data does not decay. That is consistent with the 
noise level observed in the reconstruction.}
\label{2D_star_conv3}
\end{figure}


\section{Discussion and Future Prospects}
\label{conclusion}

This work proposes some iterative regularization method evolved
from some conventional optimization algorithm. We have observed that
the choices of the parameters serve
for the convergence analysis both in the Hadamard sense and
in the iterative procedure sense. Furthermore,
we must emphasize that the theoretical development for the stability
analysis in the continous mathematical setting, \textit{i.e.} the Hadamard
convergence of the non-iterative regularized solution 
$u_{\alpha}^{\delta}$ 
to the minimum norm solution, is based on how the initial guess 
$u^{0}$ is chosen. We have been committed to make this selection 
due to mathematical
difficulties appearing in the derivation of the 
Lemma \ref{lemma_regpar_lowerbnd}, in particular on the inner products
in the proof. However, it is our belief that with a different 
non-smooth analysis such difficulties can be overcome. Much of our theoretical
results in the continuous setting
are comparable with its counterparts available in the literature
of variational regularization, see \textit{e.g.} \textbf{\cite{HofmannMathe12}}.

Regarding the stability analysis of the iterative approximation $u_{i},$
we have not needed to make any assumption on the initial guess $u_{0}.$
Instead, we have observed in Theorem \ref{IterMinVSMin}
that initialization of the inner loop, \textit{i.e.}
the initial guess for the dual variable, is important for 
stating that
$u_{i}$ is the approximation of $u_{\alpha}^{\delta}.$ 
Furthermore, the dynamical Bregman iterative form is maintained 
kby updating the dual variables.

Although we have used Morozov`s discrepancy principle, 
further work could be done on the investigation of some Lepskij 
type stopping rule. Then a new convergence
scheme is required for the iteratively regularized approximation 
$u_{i}$ towards the minimum norm solution.

In the numerical results, we have profiled the stability of the algorithm
depending on the number of the measurements, see Figures 
\ref{2D_star_W_sinogram} and \ref{2D_star_conv}. 
An open question still remains about the effect on recontruction
of the angular arrangment of ray measurement.

\newpage
\begin{appendices}
\chapter{\textbf{APPENDIX}}
\section{VSC as Upper Bound for the Bregman Distance}
\label{VSC2}

The total error estimation can also be stabilized 
due to the following assumption that has been
derived in the literature listed in Section \ref{data_smoothness},
\begin{eqnarray}
\label{total_error_bregman}
E(u_{\alpha}^{\delta} , u^{\dagger}) 
\leq D_{\mathcal{J}}(u_{\alpha}^{\delta} , u^{\dagger}).
\end{eqnarray}
Therefore, for stabilization of $E$, we seek a stable 
upper bound for the Bregman distance (\ref{total_error_bregman}).
\begin{assump}
\label{assump_conventional_variational_ineq2}
\textbf{[Variational Source Condition]}
Let $T : \mathcal{X} \rightarrow \mathcal{Y}$ be linear, injective 
forward operator and $v^{\dagger} \in \mathrm{range}(T).$ 
There exists some constant $\sigma \in (0 , 1]$
and a concave, monotonically increasing
index function $\Psi$ with $\Psi(0) = 0$ and 
$\Psi : [0 , \infty) \rightarrow [0 , \infty)$
such that for $q^{\dagger} \in \partial \mathcal{J}(u^{\dagger})$
the minimum norm solution $u^{\dagger}\in\mathrm{BV}(\Omega)$
satisfies
\begin{eqnarray}
\label{variational_ineq2}
\sigma D_{\mathcal{J}}(u,u^{\dagger}) \leq 
\mathcal{J}(u) - \mathcal{J}(u^{\dagger})
+ \Psi\left( \Vert T u - T u^{\dagger}\Vert \right)  
\mbox{, for all }
u \in \mathcal{X} .
\end{eqnarray} 
\end{assump}
Recall that quantitative stability analysis in the continuous 
mathematical setting aims to find upper bound for the total 
error estimation functional $E$ in 
(\ref{total_err_est}). According to (\ref{total_error_bregman}), 
by means of finding stable upper bound for the Bregman distance
between the regularized minimizer $u_{\alpha}^{\delta}$ and
the minimum norm solution $u^{\dagger}$ will yield one of the 
two convergence results of this section.
With the established choice of the regularization parameter and
the asserted $\mathcal{J}$ difference estiamtion in Lemma 
\ref{lemma_J_diff},
the last ingredient of the Bregman distance following up the 
Assumption \ref{assump_conventional_variational_ineq2}
is formulated below.
\begin{lemma}
\label{lemma_BregmanInnerProd}
Let $\alpha(\delta,v^{\delta})\in\overline{S}\cap\underline{S}$
be the regularization parameter for the regularized solution
$u_{\alpha}^{\delta}$ to the problem (\ref{problem0}).
If the minimum norm solution $u^{\dagger}$ satisfies 
Assumption \ref{assump_conventional_variational_ineq2}, 
then
\begin{displaymath}
- \langle D^{\ast}w^{\dagger},u_{\alpha}^{\delta} - u^{\dagger} \rangle
= \mathcal{O}(\Psi(\delta)),
\end{displaymath}
holds.
\end{lemma}
\begin{proof}
It follows from VSC (\ref{variational_ineq}) that
\begin{eqnarray}
\frac{\sigma}{2}\left(\mathcal{J}(u_{\alpha}^{\delta}) - 
\mathcal{J}(u^{\dagger}) - 
\langle D^{\ast}w^{\dagger},u_{\alpha}^{\delta} - u^{\dagger}\rangle\right)
\leq \mathcal{J}(u_{\alpha}^{\delta}) - \mathcal{J}(u^{\dagger}) + 
\Psi(\Vert T u_{\alpha}^{\delta} - T u^{\dagger}\Vert).
\end{eqnarray} 
After arranging the terms,
\begin{eqnarray}
-\langle D^{\ast}w^{\dagger},u_{\alpha}^{\delta}-u^{\dagger}\rangle 
& \leq & \frac{2}{\sigma}\left(1 - \frac{\sigma}{2}\right)
\left(\mathcal{J}(u_{\alpha}^{\delta}) - \mathcal{J}(u^{\dagger}) \right)
+ \Psi(\Vert T u_{\alpha}^{\delta} - T u^{\dagger}\Vert)
\nonumber\\
& \stackrel{(\ref{J-diff2})}{\leq} &
\frac{2}{\sigma}\left(1 - \frac{\sigma}{2}\right)
\left(\frac{1}{\underline{\tau}-1}\right)\Psi(\delta) + 
\Psi(\Vert T u_{\alpha}^{\delta} - T u^{\dagger}\Vert)
\nonumber\\
& \stackrel{(\ref{consequence_MDP})}{\leq} &
\frac{2}{\sigma}\left(1 - \frac{\sigma}{2}\right)
\left(\frac{1}{\underline{\tau}-1}\right)\Psi(\delta) +
\Psi((\overline{\tau} + 1)\delta)
\nonumber\\
& \stackrel{(\ref{concave_id})}{\leq} &
\frac{2}{\sigma}\left(1 - \frac{\sigma}{2}\right)
\left(\frac{1}{\underline{\tau}-1}\right)\Psi(\delta) +
(\overline{\tau} + 1)\Psi(\delta)
\nonumber\\
& = &\left(\frac{2}{\sigma} -1\right)
\left(\frac{1}{\underline{\tau}-1}\right)\Psi(\delta) +
(\overline{\tau} + 1)\Psi(\delta).
\label{VSC_innerprod}
\end{eqnarray}
\end{proof}

\begin{theorem}
\label{theorem_total_err2}
Let the minimum norm solution $u^{\dagger}\in\Omega$
satisfy the VSC given by Assumption
 \ref{assump_conventional_variational_ineq2}. 
Then, in the light of the assumptions of lemmata \ref{lemma_regpar_lowerbnd}, \ref{lemma_J_diff} and finally 
\ref{lemma_BregmanInnerProd}, the following estimation holds,
\begin{displaymath}
D_{\mathcal{J}}(u_{\alpha}^{\delta} , u^{\dagger}) = 
\mathcal{O}(\Psi(\delta)).
\end{displaymath}
\end{theorem}

\begin{proof}
The proof simply follows from verifying the previously established estimations
on each components of the Bregman distance as shown below,
\begin{eqnarray}
D_{\mathcal{J}}(u_{\alpha}^{\delta} , u^{\dagger}) & = &
\mathcal{J}(u_{\alpha}^{\delta}) - \mathcal{J}(u^{\dagger})- 
\langle D^{\ast}w^{\dagger},u_{\alpha}^{\delta} - u^{\dagger} \rangle
\nonumber\\
&\stackrel{(\ref{VSC_innerprod})}{\leq}&
\mathcal{J}(u_{\alpha}^{\delta}) - \mathcal{J}(u^{\dagger}) +
\left(\frac{2}{\sigma} -1\right)
\left(\frac{1}{\underline{\tau}-1}\right)\Psi(\delta) +
(\overline{\tau} + 1)\Psi(\delta)
\nonumber\\
&\stackrel{(\ref{J-diff2})}{\leq}&
\left(\frac{1}{\underline{\tau}-1}\right)\Psi(\delta)+
\left(\frac{2}{\sigma} -1\right)
\left(\frac{1}{\underline{\tau}-1}\right)\Psi(\delta) +
(\overline{\tau} + 1)\Psi(\delta).
\end{eqnarray}
\end{proof}

\section{Further Numerical Results}

In this section, we will present some further numerical results
to emphasize the condition on the step-length $\mu$ 
in Theorem \ref{IterMinVSMin}. Although the formulated condition
allows one to choose the step-length as $\mu = \frac{2}{\Vert T \Vert^2},$
we have observed divergence when we have made this choice of $\mu,$
see Figure \ref{different_mu}.

\begin{figure}[!tbp]
\includegraphics[height=4in,width=7.5in,angle=0]
{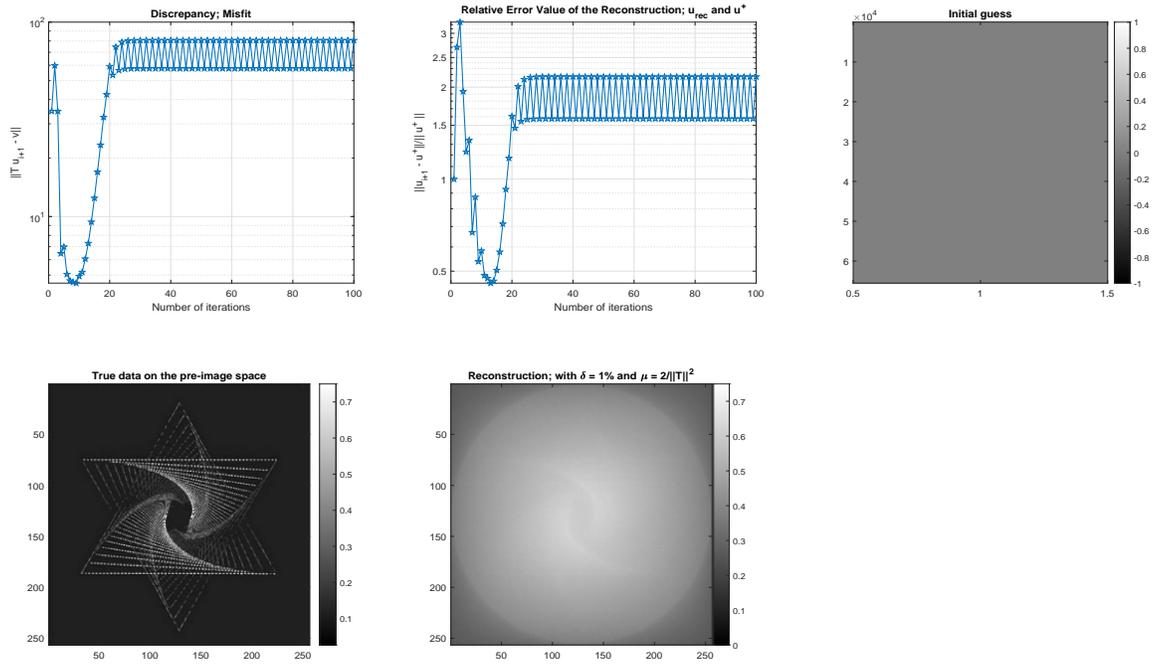}
\caption{Error analysis profiles and the data visualization:}
{\footnotesize Although the noise amount is sufficiently small 
and it is a full rank system, with some different choice of
the step-length that is defined as $\mu = \frac{2}{\Vert T\Vert^2},$
insufficient reconstruction and instability have been observed.}
\label{different_mu}
\end{figure}

\end{appendices}

\section{Acknowledgement}
The author is indebted to Ignace Loris for the fruitful discussions 
throughout the development of the work. Furthermore, the author
is highly grateful to Maria A. Gonzalez-Huici and David 
Mateos-Nunez for the encouragement and support to finalize the work. 
The work has been initiated by ARC grant at Universit\'e Libre de 
Bruxelles during author`s PostDoc research period 2017 - 2019.

\newpage
\bigskip
\section*{References}


 \bibliographystyle{alpha}

\end{document}